\documentclass[oneside,english]{amsart}
\usepackage[T1]{fontenc}
\usepackage[latin9]{inputenc}
\usepackage{amsthm}
\usepackage{amstext}
\usepackage{amssymb}
\usepackage{esint}
\usepackage{graphicx}
\usepackage{color}
\usepackage{mathrsfs}

\makeatletter
\numberwithin{equation}{section}
\numberwithin{figure}{section}
\theoremstyle{plain}
\newtheorem{thm}{\protect\theoremname}[section]
  
  \theoremstyle{plain}
  \theoremstyle{definition}
  
  \theoremstyle{plain}
  \newtheorem{lem}[thm]{\protect\lemmaname}
  \newtheorem{cor}[thm]{\protect\corollaryname}
  \theoremstyle{plain}
	\newtheorem{rem}[thm]{\protect\remarkname}
  \theoremstyle{plain}
	\theoremstyle{plain}
	
\makeatother

\usepackage{babel}
  \providecommand{\definitionname}{Definition}
  \providecommand{\lemmaname}{Lemma}
  \providecommand{\theoremname}{Theorem}
  \providecommand{\corollaryname}{Corollary}
  \providecommand{\remarkname}{Remark}
  \providecommand{\propositionname}{Proposition}
  \providecommand{\examplename}{Example}

	\DeclareMathOperator{\loc}{loc}
	\DeclareMathOperator{\adj}{adj}
	\DeclareMathOperator{\ess}{ess}
	
	\DeclareMathOperator{\cp}{cap}
  \DeclareMathOperator{\ACL}{ACL}
	\DeclareMathOperator{\dist}{dist}
  \DeclareMathOperator{\diam}{diam}

\begin{document}

\title[Geometric theory of composition operators on Sobolev spaces]{Geometric theory of composition operators on Sobolev spaces}

\author{Vladimir Gol'dshtein and Alexander Ukhlov}

\begin{abstract}
In this paper, we present the basic concepts of the geometric theory of composition operators on Sobolev spaces. The main objects of the theory are topological mappings which generate bounded embedding operators on Sobolev spaces by the composition rule. This theory is in some sense a "generalization" of the theory of quasiconformal mappings, but the theory of composition operators is oriented to its applications to the Sobolev embedding theorems, the spectral theory of elliptic operators and continuum mechanics problems.
\end{abstract}
\maketitle

\section{Introduction}

This review is devoted to the geometric theory of composition operators on Sobolev spaces. Composition (substitution) operators on Sobolev spaces arise in the Sobolev embedding theory \cite{S41} and in the quasiconformal mapping theory \cite{L71,N60}, and form the significant part of the geometric theory of the Sobolev spaces.

The geometric mapping theory of composition operators on Sobolev spaces goes back to the Reshetnyak problem (1968) on geometric characterizations of isomorphisms $\varphi^{\ast}$ of seminormed Sobolev spaces $L^1_n(\Omega)$ and $L^1_n(\widetilde{\Omega})$ generated by the composition rule $\varphi^{\ast}(f)=f\circ\varphi$. Here $\Omega$ and $\widetilde{\Omega}$ are domains in the $n$-dimensional Euclidean space $\mathbb R^n$, $n\geq 2$. In the framework of this problem, it was proven in \cite{VG75} that these isomorphisms are generated by quasiconformal mappings $\varphi:\Omega\to\widetilde{\Omega}$. 

This approach can be considered in the context of the prior knowledge about composition operators on spaces of continuous functions which are Banach algebras (see, \cite{DSh58}). The characterizations of  various structures of the space of continuous functions $C(S)$, whose isomorphisms define the topological space $S$ up to a topological mapping, were given in works by Banach, Stone, Eilenberg, Arens and Kelly, Hewitt, Gel'fand, and Kolmogorov. On another hand, M. Nakai \cite{N60} and L. Lewis \cite{L71} established an equivalence between isomorphisms of the Royden algebras and corresponding quasiconformal mappings.

Sobolev spaces are not Banach algebras and the characteristics of corresponding composition operators are completely different because of the global properties of Sobolev spaces. Consequently, the standard ideas of category theory are also inappropriate in the case of Sobolev spaces $L^1_p(\Omega)$ and $L^1_p(\widetilde{\Omega})$. In \cite{VG76} it was proven that a homeomorphism $\varphi : \Omega\to\widetilde{\Omega}$ generates an isomorphism of Sobolev spaces $L^1_p(\Omega)$ and $L^1_p(\widetilde{\Omega})$, $p>n$, if and only if $\varphi$ is a bi-Lipschitz mapping. This result was extended to the case $n-1<p<n$ in \cite{GR84} and to the case $1\leq p<n$ in \cite{M90} (see, also, \cite{V05}).

Mappings that generate isomorphisms of Besov spaces were considered in \cite{V89}, and Nikolskii-Besov spaces and Lizorkin-Triebel spaces were considered in \cite{V90}. In \cite{MS86} the theory of multipliers was applied to the change of variable problem in Sobolev spaces.

The bounded composition operators on Sobolev spaces go back to the Sobolev embedding theory \cite{S41}. In this paper, it was proven that smooth bi-Lipschitz homeomorphisms $\varphi: \Omega\to\widetilde{\Omega}$ generate a bounded operator
$$
\varphi^{\ast}: L^1_p(\widetilde{\Omega})\to L^1_p(\Omega),\,\,\,1\leq p\leq\infty,
$$
by the composition rule $\varphi^{\ast}(f)=f\circ\varphi$, and as a consequence, that Sobolev inequalities are preserved under the smooth bi-Lipschitz change of variables.

The next step in the theory of composition operators was given in \cite{GS82} where  extension operators in H\"older cusp domains were constructed using bounded composition operators from $L^1_p(\widetilde{\Omega})$ to $L^1_q(\Omega), q\leq p$, generated by reflections $\varphi:\Omega\to\widetilde{\Omega}$. This work initiated the development of the geometric analysis of Sobolev spaces in non-Lipschitz domains and required the characterization of homeomorphisms $\varphi: \Omega\to\widetilde{\Omega}$ that generate bounded operators
$$
\varphi^{\ast}: L^1_p(\widetilde{\Omega})\to L^1_q(\Omega),\,\,\,1\leq q\leq p\leq\infty.
$$
by the standard composition rule $\varphi^{\ast}(f)=f\circ\varphi$.

The partial solution of this problem, in the case $p=q$, was given in \cite{V88}. The case $q<p$ is more complicated. In this case a solution of this generalized Reshetnyak type problem was given in \cite{U93} in terms of integral characteristics of mappings of finite distortion (the limiting case $p=\infty$ was considered in \cite{GU10-1,GU10-2}) and is based on the countable-additive property of operator norms introduced in \cite{U93}. In the case $p=q=n$ these mappings coincide with quasiconformal mappings. In the general case they are called weak $(p,q)$-quasiconformal mappings. The part of the main result of \cite{U93} regarding the sufficiency of the integral condition was reproved in \cite{Kl12} with some additional conditions, and this article also included an interesting example of a $(p,q)$-quasiconformal mapping with $n<q\leq p<\infty$ which does not satisfy the Luzin $N^{-1}$-property.

In \cite{U93}, two main principles of the geometric theory of composition operators on Sobolev spaces were formulated:
\begin{itemize}
	\item The localization principle: There exist monotone countably additive set functions associated with the norms of composition operators generated by the restrictions of mappings to open subsets.
	\item The composition duality principle: Let a homeomorphism $\varphi: \Omega\to\widetilde{\Omega}$ generate a bounded composition operator
	$$
	\varphi^{\ast}: L^1_p(\widetilde{\Omega})\to L^1_q(\Omega),\,\,\,n-1< q\leq p<\infty.
	$$
	Then the inverse mapping $\varphi^{-1}: \widetilde{\Omega}\to\Omega$ generates a bounded composition operator
	$$
	\left(\varphi^{-1}\right)^{\ast}: L^1_{q'}(\Omega)\to L^1_{p'}(\widetilde{\Omega}),\,\,\,p'=\frac{p}{p-(n-1)}\leq q'=\frac{q}{q-(n-1)}.
	$$
\end{itemize}

In a number of subsequent works \cite{GGR95,GU10-2,U04,VU98,VU02,VU04,VU05} the foundations of the theory of weak $(p,q)$-quasiconformal mappings were given. The basic properties of weak $(p,q)$-quasiconformal mappings, such as the equivalence of the analytic description, in the terms of integral characteristics of mappings of finite distortion, to the capacitary description and the geometric description were proven. Within the framework of the weak inverse mapping theorem, based on the composition duality property, the regularity of inverse mappings was studied, including the limiting cases $q=n-1$ and $p=\infty$. In the recent, the theory of composition operators on Sobolev spaces has continued to develop intensively, as seen in \cite{GSU24,MU24_1, MU24_2, T15, V20,VP24}.

In the case $p=\infty$, homeomorphisms which generate bounded composition operators
$$
	\varphi^{\ast}: L^1_{\infty}(\widetilde{\Omega})\to L^1_q(\Omega),\,\,\,1\leq q\leq \infty,
$$
coincide with the Sobolev homeomorphisms of the class $L^1_q(\Omega;\widetilde{\Omega})$. The regularity of mappings inverse to Sobolev homeomorphisms in $L^1_q(\Omega;\widetilde{\Omega})$, $q\geq n-1$, was intensively studied in the last decades, see, for example, \cite{CHM,GU10-1,GU10-2,HK06,HKM06,U04}. 

We remark that in the proof of Theorem 1.2 \cite{CHM} the weak upper gradient of the inverse mapping $\varphi^{-1}$ is defined as the composition of a co-distortion function, which is a measurable function, with the inverse mapping $\varphi^{-1}$. If the mapping $\varphi$ does not possess the Luzin $N$-property, then this composition is not necessarily a measurable function. 
Therefore, the proof of the weak regularity of Sobolev homeomorphisms which is given in the work \cite{CHM} and is based on the area formula \cite{KMZ12} is incomplete. We give the detailed comments on Theorem 1.2 of \cite{CHM} in Section 4.2. So, we consider the weak regularity theorem for Sobolev homeomorphisms which possess the Luzin $N$-property \cite{GU10-1,GU10-2}.

Sobolev homeomorphisms which possess the Luzin $N^{-1}$-property, and hence transform measurable functions to measurable functions by the composition rule, are called \cite{Z69} measurable Sobolev homeomorphisms.
The approach to measurability which is based on the non-linear potential theory \cite{HM72} was suggested in \cite{U93,V89,V90}.

The applications of the geometric theory of composition operators on Sobolev spaces to the Sobolev embedding problems are based on the following  "anti-com\-mu\-ta\-tive diagram" which was suggested in \cite{GGu94} 
\[\begin{array}{rcl}
W^{1}_{p}(\widetilde{\Omega}) & \stackrel{\varphi^*}{\longrightarrow} & W^1_q(\Omega) \\[2mm]
\multicolumn{1}{c}{\downarrow} & & \multicolumn{1}{c}{\downarrow} \\[1mm]
L_s(\widetilde{\Omega}) & \stackrel{(\varphi^{-1})^*}{\longleftarrow} & L_r(\Omega)
\end{array}\]
and was used in \cite{GU09} for embeddings of weighted Sobolev spaces. 
In this diagram the operator $\varphi^{\ast}$ defined by the composition rule $\varphi^{\ast}(f)=f\circ\varphi$ is a bounded composition operator on Sobolev spaces induced by a homeomorphism $\varphi$ of $\Omega$ and $\widetilde{\Omega}$ and the operator $(\varphi^{-1})^{\ast}$ defined by the composition rule $(\varphi^{-1})^{\ast}(f)=f\circ\varphi^{-1}$ is a bounded composition operator on Lebesgue spaces. 

This method allows us to obtain Sobolev type embedding theorems in non-Lipschitz domains $\widetilde{\Omega}\subset\mathbb R^n$ that give an approach to the complicated problem of the non-Lipschitz analysis of Sobolev spaces.
By using the Min-Max Principle and Sobolev embedding theorems, we obtained applications of weak $(p,q)$-quasiconformal mappings to the spectral theory of elliptic operators (see, for example, \cite{GPU24,GHPU,GHU,GPU18_3,GU16,GU17}). In some cases 
the method based on composition operators allows one to obtain better estimates than the classical ones by L.~E.~Payne and H.~F.~Weinberger even in convex domains \cite{PW}. In a certain sense the method of composition operators on Sobolev spaces represents a remake of the classical method of conformal normalizations \cite{PS51} and gives spectral estimates in non-convex and non-Lipschitz domains even in the case of domains with fractal boundaries. 

The composition operators on Sobolev spaces have significant applications in the theory of extension (trace) operators on Sobolev spaces \cite{GS82,GU16S,KZ22}. The theory of set functions associated with composition operators was used to solve problems of Sobolev extension operators theory in \cite{KUZ22,U99,U20}.
The applications of the composition operators to the non-linear elasticity problems \cite{B81,HK04,Sv88} were given in \cite{GU23}.

Note that composition operators on Sobolev spaces admit a capacitory description \cite{GGR95,U93,VU98}  and therefore closely related to so-called $Q$-ho\-me\-o\-mor\-phisms \cite{MRSY}. The  $Q$-homeomorphisms represent a generalization of quasiconformal mappings from the point of view of the geometric function theory. The study of  $Q$-ho\-me\-o\-mor\-phisms is based on the capacitary (moduli) distortion of these classes and was under intensive development at last decades (see, for example, \cite{KRSS14,KRSS14M,K86,Sa15,S16,T15}). In the recent work \cite{MU24_2} it was proved that $Q$-homeomorphisms with integrable function $Q$, introduced in \cite{MRSY01}, coincide with mappings that generate bounded composition operators on Sobolev spaces.
\vskip 0.2cm
 
{\bf Some additional remarks.}
In the paper \cite{IOZ20} (2020) bi-Sobolev mappings $\varphi:\Omega\to\widetilde{\Omega}$ such that $\varphi\in L^1_n(\Omega)$ and $\varphi^{-1}\in L^1_n(\widetilde{\Omega})$ were considered in connections with the non-linear elasticity problems. Note that these classes were studied in \cite{OS65} as $BL$-mappings and coincide with  classes of weak $(n,1)$-quasiconformal mappings.

In \cite{GSU24}, it was proved that the class of $(p,Q)$-homeomorphisms coincides with the class of weak ($p',n-1$)-quasiconformal mappings, $p'=p/(p-n+1)$.

\medskip

In this review we give an introduction to the geometric theory of composition operators on Sobolev spaces. We consider analytical and capacitory properties of weak $(p,q)$-quasiconformal mappings due to their important applications to the elliptic operators theory. The geometric characterizations of weak $(p,q)$-quasiconformal mappings can be found in \cite{GGR95,U24,VU98}. Some review on the theory of composition operators on Sobolev spaces can be found in \cite{V12}, where the author presents his view on the development of the theory of composition operators on Sobolev spaces.

\section{Sobolev Spaces, Capacity and Composition Operators}

\subsection{Sobolev spaces  and capacity}

Let $\Omega\subset\mathbb R^n$ be an open set. The Sobolev space with first weak derivatives $W^1_p(\Omega)$, $1\leq p\leq \infty$ is a space of locally integrable weakly differentiable functions
$f:\Omega\to\mathbb{R}$ equipped with the following norm: 
\[
\|f\mid W^1_p(\Omega)\|=\|f\mid L_p(\Omega)\|+\|\nabla f \mid L_p(\Omega)\|.
\]
It is well known that the Sobolev space $W^1_p(\Omega)$ is a Banach space and coincides with the closure of the space of smooth functions $C^{\infty}(\Omega)$ in the norm of $W^1_p(\Omega)$ \cite{M}. It implies that functions of the Sobolev space $W^1_p(\Omega)$ are defined (as elements of Banach space) up to a set of $p$-capacity zero, in accordance with convergence in the topology of  $W^1_p(\Omega)$. The Sobolev space $W^{1}_{p,\loc}(\Omega)$ is defined as a space of functions $f\in W^{1}_{p}(U)$ for every open and bounded set $U\subset  \Omega$ such that $\overline{U}  \subset \Omega$.

The homogeneous seminormed Sobolev space $L^1_p(\Omega)$, $1\leq p\leq \infty$, is defined as a space
of locally integrable weakly differentiable functions $f:\Omega\to\mathbb{R}$ equipped
with the following seminorm: 
\[
\|f\mid L^1_p(\Omega)\|=\|\nabla f \mid L_p(\Omega)\|.
\]
The Sobolev space $L^{1}_{p,\loc}(\Omega)$ is defined as a space of functions $f\in L^{1}_{p}(U)$ for every open and bounded set $U\subset  \Omega$ such that $\overline{U}  \subset \Omega$.

In the case of bounded Lipschitz domains $\Omega\subset\mathbb R^n$, $n\geq 2$, Sobolev spaces $W^1_p(\Omega)$ and $L^1_p(\Omega)$ coincide because the corresponding Sobolev-Poincar\'{e} inequalities
$$
\inf\limits_{c\in\mathbb R}\|f-c \mid L_p(\Omega)\|\leq B_{p,p}(\Omega)\|f\mid L^1_p(\Omega)\|
$$
hold, see, for example, \cite{M}. Thus, Sobolev spaces $W^1_{p,\loc}(\Omega)$ and $L^1_{p,\loc}(\Omega)$ coincide for any open set $\Omega\subset\mathbb R^n$.

In the general case coincidence of the Sobolev spaces $W^1_p(\Omega)$ and $L^1_p(\Omega)$ is the complicated open problem of the geometric analysis of Sobolev spaces.

As we noted, Sobolev spaces are Banach spaces of equivalence classes \cite{M} up to a set of corresponding capacity zero. This approach is based on the concept of so-called quasicontinuity which was introduced in 70th by V.~Havin and V.~Maz'ya  \cite{HM72}. This concept is essential for a correct understanding of the theory of composition operators, otherwise the composition operator should be considered on subspaces of  continuous Sobolev functions \cite{HK12,Kl12}, are not necessarily Banach spaces.
To clarify the notion of equivalence classes of Banach Sobolev spaces we need to use the non-linear $p$-capacity associated with corresponding Sobolev spaces.

Recall the notion of the $p$-capacity of a set $E\subset \Omega$ \cite{GResh,HKM,M}. 
Suppose $\Omega$ is an open set in $\mathbb R^n$ and  $F\subset\Omega$ is a compact set. The $p$-capacity of $F$ with respect to $\Omega$ is defined as
\begin{equation*}
\cp_p(F;\Omega) =\inf_f\{\|\nabla f|L_p(\Omega)\|^p\},
\end{equation*} 
where the infimum is taken over all functions $f\in C_0(\Omega)\cap L^1_p(\Omega)$ such that $f\geq 1$ on $F$; such functions are called admissible functions for the compact set $F\subset\Omega$. 
If 
$U\subset\Omega$ 
is an open set, we define
\begin{equation*}
\cp_{p}(U;\Omega)=\sup_F\{\cp_{p}
(F;\Omega)\,:\,F\subset U,\,\, F\,\,\text{is compact}\}.
\end{equation*}

In the case of an arbitrary set 
$E\subset\Omega$
we define the inner $p$-capacity 
\begin{equation*}
\underline{\cp}_{p}(E;\Omega)=\sup_F\{\cp_{p}(F;\Omega)\, :\,\,F\subset E\subset\Omega,\,\, F\,\,\text{is compact}\},
\end{equation*}
and the outer $p$-capacity 
\begin{equation*}
\overline{\cp}_{p}(E;\Omega)=\inf_U\{\cp_{p}(U;\Omega)\, :\,\,E\subset U\subset\Omega,\,\, U\,\,\text{is open}\}.
\end{equation*}

Then a set $E\subset\Omega$ is called $p$-capacity measurable if $\underline{\cp}_p(E;\Omega)=\overline{\cp}_p(E;\Omega)$. Let $E\subset\Omega$ be a $p$-capacity measurable set. The value
$$
\cp_p(E;\Omega)=\underline{\cp}_p(E;\Omega)=\overline{\cp}_p(E;\Omega)
$$
is called the $p$-capacity measure of the set $E\subset\Omega$. By \cite{Ch54}, analytical sets (in particular Borel sets) are $p$-capacity measurable sets.

The notion of $p$-capacity permits us to refine the notion of Sobolev functions. Let a function $f\in L^1_p(\Omega)$. Then the refined function 
$$
\tilde{f}(x)=\lim\limits_{r\to 0}\frac{1}{|B(x,r)|}\int\limits_{B(x,r)} f(y)~dy
$$
is defined quasieverywhere, i.~e. up to a set of $p$-capacity zero, and it is absolutely continuous on almost all lines \cite{M}. This refined function $\tilde{f}\in L^1_p(\Omega)$ is called a unique quasicontinuous representation ({\it a canonical representation}) of function $f\in L^1_p(\Omega)$. Recall that a function $\tilde{f}$ is termed quasicontinuous if for any $\varepsilon >0$ there is an open  set $U_{\varepsilon}$ such that the $p$-capacity of $U_{\varepsilon}$ is less than $\varepsilon$ and on the set $\Omega\setminus U_{\varepsilon}$ the function  $\tilde{f}$ is continuous (see, for example, \cite{HKM,M}). 
In what follows we will use the quasicontinuous (refined) functions only.

Note that the first weak derivatives of
the function $f$ coincide almost everywhere with the usual
partial derivatives (see, e.g., \cite{M} ).

\subsection{The change of variables in the Lebesgue integral}

Let us recall the formula of the change of variables in the Lebesgue integral \cite{Fe69, Ha93} in the required for us form.
Let $\Omega,\widetilde{\Omega}$ be domains in $\mathbb R^n$ and let a homeomorphism $\varphi : \Omega\to \widetilde{\Omega}$ be such that there exists a collection of closed sets $\{A_k\}_1^{\infty}$, $A_k\subset A_{k+1}\subset \Omega$ for which restrictions $\varphi \vert_{A_k}$ are Lipschitz mapping on the sets $A_k$ and 
$$
\biggl|\Omega\setminus\bigcup\limits_{k=1}^{\infty}A_k\biggr|=0.
$$
Then there exists a Borel set $S\subset \Omega$, $|S|=0$,  such that  the mapping $\varphi:\Omega\setminus S \to \widetilde{\Omega}\setminus\widetilde{S}$, $\widetilde{S}=\varphi(S)$, has the Luzin $N$-property and the change of variables formula
\begin{equation}
\label{chvf}
\int\limits_E f\circ\varphi (x) |J(x,\varphi)|~dx=\int\limits_{\varphi(E)\setminus\widetilde{S}} f(y)~dy
\end{equation}
holds for every measurable set $E\subset \Omega$ and for every nonnegative measurable function $f: \widetilde{\Omega}\to\mathbb R$. 
If a mapping $\varphi$ possesses the Luzin $N$-property (the image of a set of measure zero has measure zero), then $|\varphi (S)|=0$ and the second integral can be rewritten as the integral on $\varphi(E)$.

\subsection{Sobolev mappings}

Sufficient conditions on the change of variables formula can be formulated in the terms of Sobolev mappings.
Let $\Omega\subset\mathbb R^n$ be an open set. Then a mapping $\varphi:\Omega\to\mathbb R^n$ belongs to the Sobolev space $L^1_{p,\loc}(\Omega)$, 
$1\leq p\leq\infty$, if its coordinate functions $\varphi_j$ belong to $L^1_{p,\loc}(\Omega)$, $j=1,\dots,n$.

The Sobolev homeomorphisms of the class $W^1_{1,\loc}(\Omega)$ satisfy the conditions of the change of variable formula \cite{Ha93} and so for the Sobolev mappings the change of variable formula \eqref{chvf} holds. Note that homeomorphisms of the class $W^1_{n,\loc}(\Omega)$ have the Luzin $N$-property \cite{VGR79}.

In the case of Sobolev mappings the formal Jacobi matrix
$$
D\varphi(x)=\left(\frac{\partial \varphi_i}{\partial x_j}(x)\right),\,\,\, i,j=1,\dots,n,
$$
and its determinant (Jacobian) $J(x,\varphi)=\det D\varphi(x)$ are well defined at
almost all points $x\in \Omega$. The norm $|D\varphi(x)|$ of the matrix
$D\varphi(x)$ is the operator norm of the corresponding linear operator $D\varphi (x):\mathbb R^n \rightarrow \mathbb R^n$ defined by the matrix $D\varphi(x)$. Recall that a mapping $\varphi$ is called a mapping of finite distortion if $|D\varphi(x)|=0$ for almost all $x\in Z=\{x\in\Omega : J(x,\varphi)=0\}$ \cite{VGR79}.

\section{Composition operators on Sobolev spaces}

Let $\Omega$ and $\widetilde{\Omega}$ be domains in $\mathbb R^n$. Then a homeomorphism $\varphi:\Omega\to\widetilde{\Omega}$ induces a bounded composition
operator 
\[
\varphi^{\ast}:L^1_p(\widetilde{\Omega})\to L^1_q(\Omega),\,\,\,1\leq q\leq p\leq\infty,
\]
by the composition rule $\varphi^{\ast}(f)=f\circ\varphi$ if the composition $\varphi^{\ast}(f)\in L^1_q(\Omega)$
is defined quasi-everywhere (almost everywhere if $q=1$) in $\Omega$ and there exists a constant $K_{p,q}(\varphi;\Omega)<\infty$ such that 
\[
\|\varphi^{\ast}(f)\mid L^1_q(\Omega)\|\leq K_{p,q}(\varphi;\Omega)\|f\mid L^1_p(\widetilde{\Omega})\|
\]
for every function $f\in L^1_p(\widetilde{\Omega})$ \cite{VU04}.

Set functions that are defined on open subsets play a significant role in the theory of composition operators on Sobolev spaces.
Recall that a nonnegative mapping $\Phi$ defined on open subsets of $\Omega$ is called a monotone countably additive set function \cite{RR55,VU04} if

\noindent
1) $\Phi(U_1)\leq \Phi(U_2)$ if $U_1\subset U_2\subset\Omega$;

\noindent
2)  for any collection $U_i \subset U \subset \Omega$, $i=1,2,...$, of mutually disjoint open sets
$$
\sum_{i=1}^{\infty}\Phi(U_i) = \Phi\left(\bigcup_{i=1}^{\infty}U_i\right).
$$

The following lemma gives properties of monotone countably additive set functions defined on open subsets of $\Omega\subset \mathbb R^n$ \cite{RR55,VU04}.

\begin{lem}
\label{lem:AddFun}
Let $\Phi$ be a monotone countably additive set function defined on open subsets of the domain $\Omega\subset \mathbb R^n$. Then

\noindent
(a) at almost all points $x\in \Omega$ there exists a finite derivative
$$
\lim\limits_{r\to 0}\frac{\Phi(B(x,r))}{|B(x,r)|}=\Phi'(x);
$$

\noindent
(b) $\Phi'(x)$ is a measurable function;

\noindent
(c) for every open set $U\subset \Omega$ the inequality
$$
\int\limits_U\Phi'(x)~dx\leq \Phi(U)
$$
holds.
\end{lem}

In the next assertion we formulate the  "localization measure principle" \cite{U93,U04} for composition operators.

\begin{thm}
\label{thm:SobolevAddFun}
 Let a mapping $\varphi : \Omega\to \widetilde{\Omega}$, where $\Omega,\,\widetilde{\Omega}$ are domains of the Euclidean space $\mathbb R^n$,
generates a bounded composition operator
$$
\varphi^{\ast} : L^1_p(\widetilde{\Omega})\to L^1_q(\Omega),\quad 1\leq q<p\leq\infty.
$$
Then
$$
\Phi(\widetilde{A})=\sup\limits_{f\in{L}_{p}^{1}(\widetilde{\Omega})\cap C_0(\widetilde{A})}
\Biggl(
\frac{\bigl\|\varphi^{\ast} f\mid {L}_{q}^{1}(\Omega)\bigr\|}
{\bigl\|f\mid {L}_{p}^{1}(\widetilde{A})\bigr\|}
\Biggr)^{\kappa}
$$
is a bounded monotone countably additive function defined on
open bounded subsets $\widetilde{A}\subset \widetilde{\Omega}$. As usual $C_0(\widetilde{A})$ is the space of continuous functions on $\widetilde{A}$ with a compact support and the number $\kappa$ is defined according to the expression $1/\kappa = 1/q -1/p$.
\end{thm}

\begin{proof}
If $\widetilde{A}_1\subset \widetilde{A}_2$ are bounded open  subsets in $\widetilde{\Omega}$, then, extending
the functions of the space $C_0(\widetilde{A}_1)$ by zero on the set $\widetilde{A}_2$, we have the inclusion
$$
\big(L_{p}^{1}(\widetilde{A}_1)\cap C_0(\widetilde{A}_1)\big)\subset
\big(L_{p}^{1}(\widetilde{A}_2)\cap C_0(\widetilde{A}_2)\big).
$$
Hence we obtain the monotonicity of the set function $\Phi$:
\begin{multline}
\Phi(\widetilde{A}_1)=\sup\limits_{f\in
L_{p}^{1}(\widetilde{A}_1)\cap C_0(\widetilde{A}_1)} \Biggl(
\frac{\bigl\|\varphi^{\ast} f\mid {L}_{q}^{1}(\Omega)\bigr\|}
{\bigl\|f\mid
L_{p}^{1}(\widetilde{A}_1)\bigr\|}
\Biggr)^{\kappa}\\
\leq\sup\limits_{f\in
L_{p}^{1}(\widetilde{A}_2)\cap C_0(\widetilde{A}_2)} \Biggl(
\frac{\bigl\|\varphi^{\ast} f\mid {L}_{q}^{1}(\Omega)\bigr\|}
{\bigl\|f\mid
L_{p}^{1}(\widetilde{A}_2)\bigr\|}
\Biggr)^{\kappa}=\Phi(\widetilde{A}_2).
\nonumber
\end{multline}

Now, let $\widetilde{A}_{i}$, $i=1,2,...$, be open disjoint subsets of $\widetilde{\Omega}$ and denote 
$\widetilde{A}_0=\bigcup\limits_{i=1}^{\infty}\widetilde{A}_i$.
For $i=1,2,...$ we consider functions $f_i\in L_{p}^{1}(\widetilde{A}_i)\cap C_0(\widetilde{A}_i)$ which simultaneously satisfy the following conditions
$$
\bigl\|\varphi^{\ast} f_i\mid{L}_{q}^{1}
(\Omega)\bigr\|\geq \bigl(\Phi(\widetilde{A}_i)\bigl(1-{\frac{\varepsilon}{2^i}}\bigr)\bigr)^{\frac{1}{\kappa}}\bigl\|f_i\mid L_{p}^{1}(\widetilde{A}_i)\bigr\|
$$
and
$$
\bigl\|f_i\mid L_{p}^{1}(\widetilde{A}_i)
\bigr\|^p=\Phi(\widetilde{A}_i)\bigl(1-{\frac{\varepsilon}{2^i}}\bigr),\,\,i=1,2,...,
$$
where $\varepsilon\in(0,1)$ is a fixed number.
Letting $g_N=\sum\limits_{i=1}^{N}f_i$ and applying the H\"older inequality (the case of equality), we obtain
\begin{multline}
\bigl\|\varphi^{\ast} g_N\mid{L}_{q}^{1}
(\Omega)\bigr\| \\ \geq
\biggl(\sum\limits_{i=1}^{N}
\left(\Phi(\widetilde{A}_i)\left(1-{\frac{\varepsilon}{2^i}}\right)\right)^{\frac{q}{\kappa}}
\bigl\|f_i\mid L_{p}^{1}(\widetilde{A}_i)
\bigr\|^q\biggr)^{1/q}\\
=
\biggl(\sum\limits_{i=1}^{N}\Phi(\widetilde{A}_i)
\left(1-{\frac{\varepsilon}{2^i}}\right)\biggr)^{\frac{1}{\kappa}}
\biggl\|g_N\mid L_{p}^{1}
\Bigl(\bigcup\limits_{i=1}^{N}\widetilde{A}_i\Bigr)\biggr\|
\\
\geq
\biggl(\sum\limits_{i=1}^{N}\Phi(\widetilde{A}_i)
-\varepsilon\Phi(\widetilde{A}_0) \biggr)^{\frac{1}{\kappa}}
\biggl\|g_N\mid L_{p}^{1}
\Bigl(\bigcup\limits_{i=1}^{N}\widetilde{A}_i\Bigr)\biggr\|
\nonumber
\end{multline}
since the sets, on which the gradients $\nabla\varphi^{\ast} f_i$
do not vanish, are disjoint. From here it follows that
$$
\Phi(\widetilde{A}_0)^{\frac{1}{\kappa}}\geq\sup\frac
{\bigl\|\varphi^{\ast} g_N\mid{L}_{q}^{1} (\Omega)\bigr\|}
{\biggl\|g_N\mid L_{p}^{1}
\Bigl(\bigcup\limits_{i=1}^{N}\widetilde{A}_i\Bigr)\biggr\|}\geq
\biggl(\sum\limits_{i=1}^{N}\Phi(\widetilde{A}_i)-\varepsilon\Phi(\widetilde{A}_0)
\biggr)^{\frac{1}{\kappa}},
$$
where the least upper bound is taken over all the
above-mentioned functions $g_N\in L_{p}^{1}
\Bigl(\bigcup\limits_{i=1}^{N}\widetilde{A}_i\Bigr)\cap C_0\Bigl(\bigcup\limits_{i=1}^{N}\widetilde{A}_i\Bigr)$. Since both
$N$ and $\varepsilon$ are arbitrary, we have
$$
\sum\limits_{i=1}^{\infty}\Phi(\widetilde{A}_i)
\leq \Phi\Bigl(\bigcup\limits_{i=1}^{\infty}\widetilde{A}_i\Bigr).
$$
The validity of the inverse inequality can be proved in a straightforward manner.
\end{proof}

\subsection{Capacitory properties of composition operators}

Let us start with the Luzin type properties of capacities for mappings which generate bounded composition operators on Sobolev spaces.

\begin{thm}
\label{theorem:CapacityDescPP}
Let a homeomorphism $\varphi :\Omega\to \widetilde{\Omega}$
generate a bounded composition operator
$$
\varphi^{\ast}: L^1_p(\widetilde{\Omega})\to L^1_p(\Omega),\,\,1\leq p<\infty.
$$
Then the inequality
$$
\cp_{p}^{1/p}(\varphi^{-1}(\widetilde{E});\Omega)
\leq K_{p,p}(\varphi;\Omega)\cp_{p}^{1/p}(\widetilde{E};\widetilde{\Omega})
$$
holds for every Borel set $\widetilde{E}\subset\widetilde{\Omega}$. 
\end{thm}

\begin{proof}

Let $F\subset E=\varphi^{-1}(\widetilde{E})$ be a compact set. Because $\varphi$ is a homeomorphism, $\widetilde{F}=\varphi(F)\subset\widetilde{E}$ is also a compact set.  Let $f\in C_0(\widetilde{\Omega})\cap L^1_p(\widetilde{\Omega)}$ be an arbitrary function such that $f\geq 1$ on $\widetilde{F}$. Then the composition $g=\varphi^{\ast}(f)$ belongs to $C_0(\Omega)\cap L^1_p({\Omega)}$, $g\geq 1$ on $F$ and 
$$
\|\varphi^{\ast}(f)\mid L^1_p(\Omega)\|\leq K_{p,p}(\varphi;\Omega)\|f \mid L^1_p(\widetilde{\Omega})\|.
$$
Since the function $g=\varphi^{\ast}(f)\in C_0(\Omega)\cap L^1_p({\Omega)}$ is an admissible function for the compact $F\subset E$, we have
$$
\cp_{p}^{1/p}(\varphi^{-1}(\widetilde{F});\Omega)\leq \|\varphi^{\ast}(f)\mid L^1_p(\Omega)\|\leq K_{p,p}(\varphi;\Omega) \|f \mid L^1_p(\widetilde{\Omega})\|.
$$
Taking infimum over all functions $f\in C_0(\widetilde{\Omega})\cap L^1_p(\widetilde{\Omega)}$ such that $f\geq 1$ on $\widetilde{F}$ we have 
$$
\cp_{p}^{1/p}(\varphi^{-1}(\widetilde{F});\Omega)
\leq K_{p,p}(\varphi;\Omega)\cp_{p}^{1/p}(\widetilde{F};\widetilde{\Omega})
$$
for any compact set $\widetilde{F}\subset \widetilde{E}\subset \widetilde{\Omega}$.
Now for the Borel set $\widetilde{E}\subset \widetilde{\Omega}$ we have (by the definition of the $p$-capacity of Borel sets)
$$
\cp_{p}^{1/p}(\widetilde{F};\widetilde{\Omega})\leq \underline{\cp}_{p}^{1/p}(\widetilde{E};\widetilde{\Omega})=\cp_{p}^{1/p}(\widetilde{E};\widetilde{\Omega}).
$$
Hence 
$$
\cp_{p}^{1/p}(\varphi^{-1}(\widetilde{F});\Omega)\leq K_{p,p}(\varphi;\Omega)\cp_{p}^{1/p}(\widetilde{E};\widetilde{\Omega}).
$$
Since $F=\varphi^{-1}(\widetilde{F})$ is an arbitrary compact set, $F\subset E$, $E$ is a Borel set as a preimage of the Borel set $\widetilde{E}$ under the homeomorphism $\varphi$, then
\begin{multline*}
\cp_{p}^{1/p}(\varphi^{-1}(\widetilde{E});\Omega)=\underline{\cp}_{p}^{1/p}(\varphi^{-1}(\widetilde{E});\Omega)=
\sup\limits_{F\subset E}\cp_{p}^{1/p}(F;\Omega)\\
\leq K_{p,p}(\varphi;\Omega)\cp_{p}^{1/p}(\widetilde{E};\widetilde{\Omega}).
\end{multline*}
\end{proof}

\begin{thm}
\label{theorem:CapacityDescPQ}
Let a homeomorphism $\varphi :\Omega\to \widetilde{\Omega}$
generate a bounded composition operator
$$
\varphi^{\ast}: L^1_p(\widetilde{\Omega})\to L^1_q(\Omega),\,\,1\leq q< p<\infty.
$$
Then the inequality
$$
\cp_{q}^{1/q}(\varphi^{-1}(\widetilde{E});\Omega)
\leq K_{p,q}(\varphi;\Omega)\cp_{p}^{1/p}(\widetilde{E};\widetilde{\Omega})
$$
holds for every Borel set $\widetilde{E}\subset\widetilde{\Omega}$. 
\end{thm}

\begin{proof}

Let $F\subset E=\varphi^{-1}(\widetilde{E})$ be a compact set. Because $\varphi$ is a homeomorphism, $\widetilde{F}=\varphi(F)\subset\widetilde{E}$ is also a compact set.  Let $f\in C_0(\widetilde{\Omega})\cap L^1_p(\widetilde{\Omega)}$ be an arbitrary function such that $f\geq 1$ on $\widetilde{F}$. Then the composition $g=\varphi^{\ast}(f)$ belongs to $C_0(\Omega)\cap L^1_q({\Omega)}$, $g\geq 1$ on $F$ and 
$$
\|\varphi^{\ast}(f)\mid L^1_q(\Omega)\|\leq K_{p,q}(\varphi;\Omega) \|f \mid L^1_p(\widetilde{\Omega})\|.
$$
Since the function $g=\varphi^{\ast}(f)\in C_0(\Omega)\cap L^1_q({\Omega)}$ is an admissible function for the compact $F\subset E$, then
$$
\cp_{q}^{1/q}(\varphi^{-1}(\widetilde{F});\Omega)\leq \|\varphi^{\ast}(f)\mid L^1_q(\Omega)\|\leq K_{p,q}(\varphi;\Omega) \|f \mid L^1_p(\widetilde{\Omega})\|.
$$
Taking infimum over all functions $f\in C_0(\widetilde{\Omega})\cap L^1_p(\widetilde{\Omega)}$ such that $f\geq 1$ on $\widetilde{F}$ we have 
$$
\cp_{q}^{1/q}(\varphi^{-1}(\widetilde{F});\Omega)
\leq K_{p,q}(\varphi;\Omega)\cp_{p}^{1/p}(\widetilde{F};\widetilde{\Omega})
$$
for any compact set $\widetilde{F}\subset \widetilde{E}\subset \widetilde{\Omega}$.

Now for the Borel set $\widetilde{E}\subset \widetilde{\Omega}$ we have (by the definition of the $p$-capacity of Borel sets)
$$
\cp_{p}^{1/p}(\widetilde{F};\widetilde{\Omega})\leq \underline{\cp}_{p}^{1/p}(\widetilde{E};\widetilde{\Omega})=\cp_{p}^{1/p}(\widetilde{E};\widetilde{\Omega}).
$$
Hence 
$$
\cp_{q}^{1/q}(\varphi^{-1}(\widetilde{F});\Omega)\leq K_{p,q}(\varphi;\Omega)\cp_{p}^{1/p}(\widetilde{E};\widetilde{\Omega}).
$$
Since $F=\varphi^{-1}(\widetilde{F})$ is an arbitrary compact set, $F\subset E$, $E$ is a Borel set as a preimage of the Borel set $\widetilde{E}$ under the homeomorphism $\varphi$, then
\begin{multline*}
\cp_{q}^{1/q}(\varphi^{-1}(\widetilde{E});\Omega)=\underline{\cp}_{q}^{1/q}(\varphi^{-1}(\widetilde{E});\Omega)=
\sup\limits_{F\subset E}\cp_{q}^{1/q}(F;\Omega)\\
\leq K_{p,q}(\varphi;\Omega)\cp_{p}^{1/p}(\widetilde{E};\widetilde{\Omega}).
\end{multline*}

\end{proof}

\begin{cor}
\label{cor:Capacity}
Let a homeomorphism $\varphi :\Omega\to \widetilde{\Omega}$
generate a bounded composition operator
$$
\varphi^{\ast}: L^1_p(\widetilde{\Omega})\to L^1_q(\Omega), \,\,1<q\leq p<\infty.
$$
Then the preimage of a set of $p$-capacity zero has $q$-capacity zero.
\end{cor}

\begin{proof}
The proof follows from Theorem~\ref{theorem:CapacityDescPP} and Theorem~\ref{theorem:CapacityDescPQ}.

\end{proof}

The characterization of mappings which generate bounded composition operators on Sobolev spaces in capacitary terms is given in the terms of the variational $p$-capacity \cite{GResh}. The condenser in the domain $\Omega\subset \mathbb R^n$ is the pair $(F_0,F_1)$ of connected sets $F_0, F_1 \subset \Omega$ that are closed relative to $\Omega$. A continuous function $u\in L_p^1(\Omega)$ is called an admissible function for the condenser $(F_0,F_1)$ if the set $F_i\cap \Omega$ is contained in some connected component of the set $\operatorname{Int}\{x\vert u(x)=i\}$,\ $i=0,1$. We say that the $p$-capacity of the condenser $(F_0, F_1)$ relative to the domain $\Omega$ is the value
$$
{{\cp}}_p(F_0,F_1;\Omega)=\inf\|u\vert L_p^1(\Omega)\|^p,
$$
where the infimum is taken over all functions admissible for the condenser $(F_0, F_1) \subset \Omega$. If the condenser has no admissible functions, we set the capacity to equal infinity. 

Let ${{\cp}}_p(F_0,F_1;\Omega)<\infty$. Suppose that a function $v$ belongs to the Sobolev space $L_p^1(\Omega)$ is admissible for the condenser $(F_0,F_1)$. Then $v$ is called {\it an extremal function for the condenser $(F_0,F_1)$} \cite{VG77}, if
$$
\int\limits_{\Omega\backslash(F_0\cup F_1)}|\nabla v|^p\,dx=
{{\cp}}_p(F_0,F_1;\Omega).
$$
Note that for any $1<p<\infty$ and any condenser $(F_0,F_1)$ with ${{\cp}}_p(F_0,F_1;\Omega)<\infty$ the extremal function exists and unique.

The set of extremal functions for $p$-capacity of every possible pairs of $n$-dimensional connected compact sets 
$F_0,\;F_1\subset \Omega$, having smooth boundaries, we denote by the symbol $E_p(\Omega)$.

\begin{thm}
\label{thm:ApproxExrem}
\cite{VG77}
Let $1<p<\infty$. Then there exists a countable collection of functions $v_i\in E_p(\Omega)$, $i\in \mathbb N$,
such that for every function $u\in L_{p}^{1}(\Omega)$ and for any $\varepsilon>0$ there exists a representation of $u$ in the form
$u=c_0+\sum\limits_{i=1}^{\infty}c_iv_i$,
for which the inequalities 
$$
\|u\mid L_{p}^{1}(\Omega)\|^p\leq \sum\limits_{i=1}^{\infty}
\|c_iv_i\mid L_{p}^{1}(\Omega)\|^p\leq
        \|u\mid L_{p}^{1}(\Omega)\|^p+\varepsilon
$$ 
hold.
\end{thm}

The following assertion was proved, but not formulated in  \cite{VG77}.

\begin{lem}
\label{lem:extremal} Let $1<q\leq p<\infty$. Suppose that $\varphi :\Omega\to \widetilde{\Omega}$ is a homeomorphism of domains $\Omega,\widetilde{\Omega}\subset\mathbb R^n$, such that for a pair of $n$-dimensional connected compact sets $\widetilde{F}_0,\widetilde{F}_1\subset \widetilde{\Omega}$ with smooth boundaries the inequality
$$
\left(\int\limits_{\Omega}|\nabla v|^q~dx\right)^{\frac{1}{q}}\leq C_{p,q}(\widetilde{F}_0,\widetilde{F}_1) \left(\int\limits_{\widetilde{\Omega}}|\nabla u|^p~dx\right)^{\frac{1}{p}}
$$
holds for extremal functions $v$ and $u$ of pairs $(\varphi^{-1}(\widetilde{F}_0),\varphi^{-1}(\widetilde{F}_1))$ and $(\widetilde{F}_0,\widetilde{F}_1)$ correspondly, where $C_{p,q}(\widetilde{F}_0,\widetilde{F}_1)$ depends on $\widetilde{F}_0,\widetilde{F}_1\subset\mathbb R^n$. Then
$$
\|\varphi^{\ast}(u)\mid L^1_q(\Omega)\|\leq C_{p,q}(\widetilde{F}_0,\widetilde{F}_1)\|u \mid L^1_p(\widetilde{\Omega})\|.
$$
\end{lem}

\begin{rem}
The proof of this assertion in \cite{VG77} is based on the monotonicity of the capacity and the lower semi-continuity of Sobolev norms, see also  \cite{VT21}.
\end{rem}

\begin{rem}
The quantity $C_{p,q}(\widetilde{F}_0,\widetilde{F}_1)$ is the norm of the composition operator $\varphi^{\ast}$ restriction on the set of functions of $E_p(\widetilde{\Omega}))$ with the gradients $\nabla u$ vanishing on the set $\partial\widetilde{F}_0\cup\partial\widetilde{F}_1$.
\end{rem}

The next assertion gives the capacitory description of composition operators on Sobolev spaces $L^1_p(\widetilde{\Omega})$ and $L^1_p(\Omega)$. In the proof we follow \cite{U93}.

\begin{thm}
\label{thm:CapacityDescPP}
Let $\Omega$ and $\widetilde{\Omega}$ be domains in $\mathbb R^n$. Then a homeomorphism $\varphi :\Omega\to \widetilde{\Omega}$
generates a bounded composition operator
$$
\varphi^{\ast}: L^1_p(\widetilde{\Omega})\to L^1_p(\Omega),\,\,1<p<\infty,
$$
if and only if there exists a constant  $K_{p,p}(\varphi;\Omega)$ such that for every condenser 
$(F_0,F_1)\subset \Omega$
the inequality
$$
\cp_{p}^{1/p}(\varphi^{-1}(F_0),\varphi^{-1}(F_1);\Omega) \leq K_{p,p}(\varphi;\Omega)\cp_{p}^{1/p}(F_0,F_1;\widetilde{\Omega})
$$
holds 
\end{thm}

\begin{proof}
{\it Necessity.}
Let $u$ be an admissible function for the condenser $(F_0,F_1)\subset\widetilde{\Omega}$.
Then $u\circ\varphi$ is an admissible function for the condenser
$(\varphi^{-1}(F_0),\varphi^{-1}(F_1))\subset\Omega$.
Since $\varphi$
generates a bounded composition operator
$$
\varphi^{\ast}: L^1_p(\widetilde{\Omega})\to L^1_p(\Omega)
$$
then, by the definition of the $p$-capacity, we have
$$
{\cp}_p^{1/p}(\varphi^{-1}(F_0),\varphi^{-1}(F_1);\Omega)\leq\|\varphi^{\ast}(u)\vert L_q^1(\Omega)\|\leq K_{p,p}(\varphi;\Omega)\|u\,|\,L_p^1(\widetilde{\Omega})\|.
$$
Since $u$ is an arbitrary admissible function, we obtain
$$
\cp_{p}^{1/p}(\varphi^{-1}(F_0),\varphi^{-1}(F_1);\Omega)
\leq K_{p,p}(\varphi;\Omega)\cp_{p}^{1/p}(F_0,F_1;\widetilde{\Omega}).
$$

{\it Sufficiency.}
Let a function $u\in L^1_p(\widetilde{\Omega})$. Then by Theorem~\ref{thm:ApproxExrem} 
$$
u=c_0+\sum\limits_{i=1}^{\infty}c_i v_i,\,\,v_i\in E_p(\widetilde{\Omega}),
$$
and
$$
\|u\mid L_{p}^{1}(\widetilde{\Omega})\|^p\leq \sum\limits_{i=1}^{\infty}\|c_iv_i\mid L_{p}^{1}(\widetilde{\Omega})\|^p\leq
\|u\mid L_{p}^{1}(\widetilde{\Omega})\|^p+\varepsilon
$$ 
for arbitrary $\varepsilon>0$.

The series $c_0+\sum\limits_{i=1}^{\infty}c_i v_i$ converges to the function $u\in L^1_p(\widetilde{\Omega})$ $p$-quasieverywhere (up to a set of $p$-capacity zero) in $\widetilde{\Omega}$. Hence, by Corollary~\ref{cor:Capacity} we have that the series $c_0+\sum\limits_{i=1}^{\infty}c_i (v_i\circ\varphi)$ converges to the function $u\circ\varphi=\varphi^{\ast}(u)$ $p$-quasieverywhere  (up to a set of $p$-capacity zero) in $\Omega$.

By Lemma~\ref{lem:extremal} we obtain
\begin{multline}
\|(c_0+\sum\limits_{i=1}^\infty c_i \varphi^{\ast}(v_i))\vert L_p^1(\Omega)\|^p=
\sum\limits_{i=1}^\infty c_i^p \| \varphi^{\ast}(v_i)\vert L_p^1(\Omega)\|^p\\
\le \sum\limits_{i=1}^\infty C_{p,p}^p(\widetilde{F}_0^i,\widetilde{F}_1^i)c_i^p \| v_i \vert L_p^1(\widetilde{\Omega})\|^p
\leq K_{p,p}^p(\varphi;\Omega) \sum\limits_{i=1}^\infty \|c_i v_i \vert L_p^1(\widetilde{\Omega})\|^p\\
=K_{p,p}^p(\varphi;\Omega) \|u \vert L_p^1(\widetilde{\Omega})\|^p+\varepsilon K_{p,p}^p(\varphi;\Omega).
\nonumber
\end{multline}
Since $\varepsilon>0$ is arbitrary, we obtain
that for every function $u\in L_p^1(\widetilde{\Omega})$
$$
\|\varphi ^{\ast}(u)\vert L_p^1(\Omega)\|\le K_{p,p}(\varphi;\Omega)\|u\vert L_p^1(\widetilde{\Omega})\|.
$$
\end{proof}

The case $1<q<p<\infty$ was given in \cite{U93}.

\begin{thm}
\label{thm:CapacityDescPQ}
Let $\Omega$ and $\widetilde{\Omega}$ be domains in $\mathbb R^n$. Then a homeomorphism $\varphi :\Omega\to \widetilde{\Omega}$
generates a bounded composition operator
$$
\varphi^{\ast}: L^1_p(\widetilde{\Omega})\to L^1_q(\Omega),\,\,1<q<p<\infty,
$$
if and only if
there exists a bounded monotone countable-additive set function
$\Phi$ defined on open subsets of $\Omega$
such that for every condenser 
$(F_0,F_1)\subset \Omega$
the inequality
$$
\cp_{q}^{1/q}(\varphi^{-1}(F_0),\varphi^{-1}(F_1);\Omega)
\leq\Phi(\widetilde{\Omega}\setminus(F_0\cup F_1))^{\frac{p-q}{pq}}
\cp_{p}^{1/p}(F_0,F_1;\widetilde{\Omega})
$$
holds. 
\end{thm}

\begin{proof}
{\it Necessity.}
Let $u$ be an admissible function for the condenser $(F_0,F_1)\subset\widetilde{\Omega}$.
Then $u\circ\varphi$ is an admissible function for the condenser
$(\varphi^{-1}(F_0),\varphi^{-1}(F_1))\subset\Omega$.
Since $\varphi$
generates a bounded composition operator
$$
\varphi^{\ast}: L^1_p(\widetilde{\Omega})\to L^1_q(\Omega)
$$
then, by the definition of the $p$-capacity, we have
\begin{multline}
{\cp}_q^{1/q}(\varphi^{-1}(F_0),\varphi^{-1}(F_1);\Omega)\leq
\|\varphi^{\ast}(u)\vert L_q^1(\Omega)\|\\
\le
\Phi(\widetilde{\Omega}\setminus(F_0\cup F_1))^{(p-q)/pq}\|u\,|\,L_p^1(\widetilde{\Omega})\|,
\nonumber
\end{multline}
where the set function $\Phi$  is defined in Theorem~\ref{thm:SobolevAddFun}.

Since $u$ is an arbitrary admissible function, we have
$$
{\cp}_q^{1/q}(\varphi^{-1}(F_0),\varphi^{-1}(F_1);\Omega)
\le
\Phi(\widetilde{\Omega}\setminus(F_0\cup F_1))^{(p-q)/pq}{\cp}_p^{1/p}(F_0,F_1;\widetilde{\Omega}).
$$

{\it Sufficiency.}
Let $u\in L^1_p(\widetilde{\Omega})$. Then by Theorem~\ref{thm:ApproxExrem} 
$$
u=c_0+\sum\limits_{i=1}^{\infty}c_i v_i,\,\,v_i\in E_p(\widetilde{\Omega}),
$$
and
$$
\|u\mid L_{p}^{1}(\widetilde{\Omega})\|^p\leq \sum\limits_{i=1}^{\infty}\|c_iv_i\mid L_{p}^{1}(\widetilde{\Omega})\|^p\leq
\|u\mid L_{p}^{1}(\widetilde{\Omega})\|^p+\varepsilon
$$ 
for arbitrary $\varepsilon>0$.

The series $c_0+\sum\limits_{i=1}^{\infty}c_i v_i$ converges to the function $u\in L^1_p(\widetilde{\Omega})$ $p$-quasieverywhere (up to a set of $p$-capacity zero) in $\widetilde{\Omega}$. Hence, by Corollary~\ref{cor:Capacity} we have that the series $c_0+\sum\limits_{i=1}^{\infty}c_i (v_i\circ\varphi)$ converges to the function $u\circ\varphi=\varphi^{\ast}(u)$ $q$-quasieverywhere  (up to a set of $q$-capacity zero) in $\Omega$.

We denote by $Q_{p,q}(\varphi;\Omega)$ the total variation of the set function $\Phi$ raised to the power $(p-q)/pq$: 
$$
Q_{p,q}(\varphi;\Omega) = \sup\limits_{\{\widetilde A_k\}} \left( \sum\limits_{k=1}\limits^{\infty} {\Phi}(\widetilde A_k)\right)^{\frac{p-q}{pq}}<\infty,
$$
where the supremum is taken over all families of disjoint open sets $\{\widetilde A_k\}_{k\in \mathbb{N}}$, $\widetilde A_k \subset \widetilde\Omega$.
By Lemma~\ref{lem:extremal} we obtain
\begin{multline}
\|(c_0+\sum\limits_{i=1}^\infty c_i \varphi^{\ast}(v_i))\vert L_q^1(\Omega)\|^q=
\sum\limits_{i=1}^\infty c_i^q \| \varphi^{\ast}(v_i)\vert L_q^1(\Omega)\|^q\\
\le \sum\limits_{i=1}^\infty C_{p,q}^q(\widetilde{F}_0^i,\widetilde{F}_1^i)c_i^q \| v_i \vert L_p^1(\widetilde{\Omega})\|^q\leq 
\sum\limits_{i=1}^\infty \Phi(\widetilde{\Omega}\setminus(F_0^i\cup F_1^i))^{\frac{p-q}{p}}c_i^q \| v_i \vert L_p^1(\widetilde{\Omega})\|^q\\
\leq\left(\sum\limits_{i=1}^\infty \Phi(\widetilde{\Omega}\setminus(F_0^i\cup F_1^i))\right)^{\frac{p-q}{p}} 
\left(\sum\limits_{i=1}^\infty c_i^p\| v_i \vert L_p^1(\widetilde{\Omega})\|^p\right)^{\frac{q}{p}}\\
\leq Q_{p,q}^{q}(\varphi;\Omega) \left(\sum\limits_{i=1}^\infty \|c_i v_i \vert L_p^1(\widetilde{\Omega})\|^p\right)^{\frac{q}{p}}\\
= Q_{p,q}^{q}(\varphi;\Omega) \left(\|u\vert L_p^1(\widetilde{\Omega})\|^p+\varepsilon\right)^{\frac{q}{p}}.
\nonumber
\end{multline}
Since $\varepsilon>0$ is arbitrary, we obtain
that for every function $u\in L_p^1(\widetilde{\Omega})$
$$
\|\varphi ^{\ast}(u)\vert L_q^1(\Omega)\|\leq Q_{p,q}(\varphi;\Omega)\|u\vert L_p^1(\widetilde{\Omega})\|.
$$
\end{proof}

\subsection{The analytical description of composition operators}

We begin with an analytical description of homeomorphisms which generate bounded composition operators on Sobolev spaces $L^1_p(\widetilde{\Omega})$ and $L^1_p(\Omega)$. In the proof we follow  \cite{U93,V89}.

\begin{thm}
\label{thm:AnalyticPP}
Let $\Omega$ and $\widetilde{\Omega}$ be domains in $\mathbb R^n$. Then a homeomorphism $\varphi :\Omega\to \widetilde{\Omega}$ generates a bounded composition operator
$$
\varphi^{\ast} : L^1_p(\widetilde{\Omega})\to L^1_p(\Omega), \,\,\,1\leq p<\infty,
$$
if and only if $\varphi\in L^1_{p,\loc}(\Omega)$ has finite distortion and
$$
\ess\sup\limits_{x\in\Omega}\left(\frac{|D\varphi(x)|^p}{|J(x,\varphi)|}\right)^{\frac{1}{p}}=K_{p,p}(\varphi;\Omega)<\infty.
$$
The norm of the operator satisfies $\|\varphi^{\ast}\|\leq K_{p,p}(\varphi;\Omega)$.
\end{thm}

\begin{proof}
{\it Necessity.} 
Fix a cutoff function $\eta\in C_{0}^{\infty}(\mathbb R^n)$, which is equal to one on the ball $B(0,1)$ and is equal to zero outside of the ball $B(0,2)$. We consider the test functions
\begin{equation}
\label{test1}
f_j(y)=(y_j-y_{j0})\eta\left(\frac{y-y_0}{r}\right),\,\,\, j=1,...,n,
\end{equation}
where $y_j$ is the $j$-th coordinate function. We see that 
\begin{multline}
\label{test2}
|\nabla( f_j(y))|=\left|(\nabla(y_j-y_{j0}))\eta\left(\frac{y-y_0}{r}\right)+(y_j-y_{j0})\nabla\biggl(\eta\left(\frac{y-y_0}{r}\right)\biggr)\right|\\
=
\left|\eta\left(\frac{y-y_0}{r}\right)+\frac{y_j-y_{j0}}{r}(\nabla\eta)\left(\frac{y-y_0}{r}\right)\right|\leq 1+2|\nabla\eta|\left(\frac{y-y_0}{r}\right)\leq C_{\eta},
\end{multline}
where 
\begin{equation}
\label{test3}
C_{\eta}=1+2\max_{z\in\mathbb R^n}|\nabla\eta|(z)<\infty.
\end{equation}

By substituting the test functions 
$$
f_j(y)=(y_j-y_{j0})\eta\left(\frac{y-y_0}{r}\right),\,\,\, j=1,...,n,
$$
into the inequality
$$
\|\varphi^{\ast}f \mid L^1_p(\Omega)\|\leq K_{p,p}(\varphi;\Omega) \|f\mid L^1_p(\widetilde{\Omega})\|,
$$
we have that
\begin{equation}\label{EqFunPP}
\biggl(\int\limits_{\varphi^{-1}(B(y_0,r))}|D\varphi(x)|^p\,dx\biggr)^{1/p}\leq CK_{p,p}(\varphi;\Omega)(r^n)^{1/p},
\end{equation}
where $C$ is a constant which depends on $n$ and $p$ only. Hence, the homeomorphism $\varphi$ belongs to the Sobolev space $L_{p,{\loc}}^1(\Omega)$. 

We then prove that $\varphi:\Omega\to\widetilde{\Omega}$ is a homeomorphism of finite distortion. Let $Z=\{x\in \Omega : J(x,\varphi)=0\}$. We prove that
$$
\int\limits_{Z}|D\varphi(x)|^p\,dx=0.
$$
In order to achieve this, we rewrite the integral as sum of two integrals:
$$
\int\limits_{Z}|D\varphi(x)|^p\,dx
=\int\limits_{Z\setminus S}|D\varphi(x)|^p\,dx+
\int\limits_{Z \cap S}|D\varphi(x)|^p\,dx,
$$
where $S$ is the set from the change of variables formula (\ref{chvf}) on which the homeomorphism $\varphi$ does not have the Luzin $N$-property.

Because $|S|=0$, we have 
$$
\int\limits_{Z\cap S}|D\varphi(x)|^p\,dx=0.
$$

Let us show that
$$
\int\limits_{Z\setminus S}|D\varphi(x)|^p\,dx=0.
$$
By the change of variable formula we have $|\varphi(Z\setminus S)|=0$. Fix $\varepsilon>0$. Then there exists a family of balls $\{B(y_i,r_i)\}$ generating a covering of the set $\varphi(Z\setminus S)$ such that the multiplicity of the covering $B(y_i,2r_i)$ is finite and $\sum\limits_i|B(y_i,r_i)|<\varepsilon$.
Then by inequality (\ref{EqFunPP}) we obtain
$$
\int\limits_{Z\setminus S}|D\varphi(x)|^p\,dx\leq
\sum\limits_{i=1}^{\infty}\int\limits_{\varphi^{-1}(B(y_i,r_i))}|D\varphi(x)|^p\,dx
\leq C^p K^p_{p,p}(\varphi;\Omega) \sum\limits_{i=1}^{\infty}|B_i|.
$$
Since $\varepsilon$ is an arbitrary number then $\int\limits_{Z\setminus S}|D\varphi|^p\,dx=0$. Hence, $D\varphi=0$ a.e. on $Z\setminus S$, and as a consequence, $D\varphi=0$ a.e. on $Z$ and the homeomorphism $\varphi$ has finite distortion.

We apply to the left side of the inequality~(\ref{EqFunPP}) the change of variable formula, and we denote $B:=B(y_0,r)$:
\begin{multline}
\biggl(\int\limits_{\varphi^{-1}(B)}|D\varphi(x)|^p\,dx\biggr)^{\frac{1}{p}}
=
\biggl(\int\limits_{\varphi^{-1}(B)\setminus S}|D\varphi(x)|^p\,dx\biggr)^{\frac{1}{p}}\\
=\biggl(\int\limits_{\varphi^{-1}(B)\setminus (S\cup Z)}
|D\varphi(x)|^p\,dx\biggr)^{\frac{1}{p}}
=
\biggl(\int\limits_{\varphi^{-1}(B)\setminus (S\cup Z)}
\frac{|D\varphi(x)|^p}{|J(x,\varphi)|}|J(x,\varphi)|\,dx\biggr)^{\frac{1}{p}}\\
=\biggl(\int\limits_{B\setminus \varphi(S)}
\frac{|D\varphi(\varphi^{-1}(y))|^p}{|J(\varphi^{-1}(y),\varphi)|}\,dy\biggr)^{\frac{1}{p}}
\leq CK_{p,p}(\varphi;\Omega)(r^n)^{\frac{1}{p}}.
\nonumber
\end{multline}
Hence, we have the inequality
$$
\biggl(\frac{1}{r^n}\int\limits_{B\setminus \varphi(S)}
\frac{|D\varphi(\varphi^{-1}(y))|^p}{|J(\varphi^{-1}(y),\varphi)|}\,dy\biggr)^{\frac{1}{p}}
\leq CK_{p,p}(\varphi;\Omega).
$$

By using the Lebesgue theorem about differentiability of the integral by measure (see, for example, \cite{RR55,VU04}) we obtain
$$
\biggl(
\frac{|D\varphi(\varphi^{-1}(y))|^p}{|J(\varphi^{-1}(y),\varphi)|}\biggr)^{\frac{1}{p}}
\leq CK_{p,p}(\varphi;\Omega)\quad\text{ for almost all}\,\,\, y\in\widetilde{\Omega}\setminus\varphi(S).
$$
Because the homeomorphism $\varphi$ has the Luzin $N$-property on the set $\Omega\setminus S$, we finally have:
$$
\ess\sup\limits_{x\in\Omega}\left(\frac{|D\varphi(x)|^p}{|J(x,\varphi)|}\right)^{\frac{1}{p}}\leq CK_{p,p}(\varphi;\Omega)<\infty.
$$

{\it Sufficiency.}
Let $f\in L_p^1(\widetilde{\Omega})\cap C^\infty(\widetilde{\Omega})$. Then the composition $f\circ\varphi$
belongs to the class $W^1_{1,\loc}(\Omega)$ and the chain rule holds \cite{Zi}. Thus, we have
\begin{multline}
\|\varphi^* f\vert L_p^1(\Omega)\|=\biggl(\int\limits_\Omega|\nabla (f\circ\varphi(x))|^p\,dx\biggr)^{\frac{1}{p}}
\leq
\biggl(\int\limits_\Omega(|\nabla f|(\varphi(x))|D\varphi(x)|)^p\,dx\biggr)^{\frac{1}{p}}\\
=\biggl(\int\limits_\Omega|\nabla f|^p(\varphi(x))|J(x,\varphi)|
\frac{|D\varphi(x)|)^p}{|J(x,\varphi)|}\,dx\biggr)^{\frac{1}{p}}.
\nonumber
\end{multline}
Hence
\begin{multline}
\|\varphi^*f\vert L_p^1(\Omega)\|=\biggl(\int\limits_\Omega|\nabla (f\circ\varphi(x))|^p\,dx\biggr)^{\frac{1}{p}}\\
\leq
\ess\sup\limits_{x\in\Omega}\left(\frac{|D\varphi(x)|^p}{|J(x,\varphi)|}\right)^{\frac{1}{p}}\biggl(\int\limits_\Omega|\nabla f|^p(\varphi(x))|J(x,\varphi)|\,dx\biggr)^{\frac{1}{p}}.
\nonumber
\end{multline}

Applying the change of variable formula gives the required inequality
$$
\|\varphi^{\ast}f \mid L^1_p(\Omega)\|\leq K_{p,p}(\varphi;\Omega) \|f \mid L^1_p(\widetilde{\Omega})\|
$$
for every smooth function $f\in L^1_p(\widetilde{\Omega})$.

To extend the estimate onto all functions $f\in L^1_p(\widetilde{\Omega})$, $1<q< p<\infty$, consider a sequence of smooth functions $f_k\in L^1_p(\widetilde{\Omega})$, $k=1,2,...$, such that $f_k\to f$ in $L^1_p(\widetilde{\Omega})$ and $f_k\to f$ $p$-quasi-everywhere in $\widetilde{\Omega}$ as $k\to\infty$. Since by Corollary~\ref{cor:Capacity} the preimage $\varphi^{-1}(S)$ of the set $S\subset \widetilde{\Omega}$ of $p$-capacity zero has the $p$-capacity zero, we have $\varphi^{\ast}(f_k)\to \varphi^{\ast}(f)$
$p$-quasi-everywhere in $\Omega$ as $k\to\infty$. This observation leads us to the following conclusion: Extension by continuity of the operator $\varphi^{\ast}$ $L^1_p(\widetilde{\Omega})\cap C^{\infty}(\widetilde{\Omega})$ to $L^1_p(\widetilde{\Omega})$ coincides with the composition operator $\varphi^{\ast}$, $\varphi^{\ast}(f) = f\circ\varphi$.

If $p=1$ then we should replace the capacitary characteristic of convergence with a coarser
one: if a sequence $f_n\in L^1_1(\widetilde{\Omega})$ converges to $f\in L^1_1(\widetilde{\Omega})$ in $L^1_1(\widetilde{\Omega})$ then a subsequence of $f_n$ converges to $f$ almost everywhere. To complete the proof, it suffices to use the following property: the inverse image of a set of measure zero under the mapping $\varphi:\Omega\to\widetilde{\Omega}$ inducing the bounded operator
$\varphi^{\ast}:L^1_1(\widetilde{\Omega})\to L^1_1(\Omega)$ is a set of measure zero.

\end{proof}

The next theorem gives the characterization of composition operators from Sobolev spaces $L^1_p(\widetilde{\Omega})$ to $L^1_q(\Omega)$ for $q<p$ \cite{U93}.

\begin{thm}
\label{thm:AnalyticPQ}
Let $\Omega$ and $\widetilde{\Omega}$ be domains in $\mathbb R^n$. Then a homeomorphism $\varphi :\Omega\to \widetilde{\Omega}$ generates a bounded composition operator
$$
\varphi^{\ast} : L^1_p(\widetilde{\Omega})\to L^1_q(\Omega), \,\,\,1\leq q< p<\infty,
$$
if and only if $\varphi\in L^1_{q,\loc}(\Omega)$ has finite distortion and
$$
\biggl(\int\limits_\Omega\bigg(\frac{|D\varphi(x)|^p}{|J(x,\varphi)|}\biggr)^{\frac{q}{p-q}}~dx\biggr)^{\frac{p-q}{pq}}=K_{p,q}(\varphi;\Omega)<\infty.
$$
The norm of the operator satisfies $\|\varphi^{\ast}\|\leq K_{p,q}(\varphi;\Omega)$.
\end{thm}

\begin{proof}
{\it Necessity.} By Theorem \ref{thm:SobolevAddFun} the inequality
\begin{equation}
\label{eqphi} 
\|\varphi^{\ast} f|{L}_{q}^{1}(\Omega)\|\leq
\Phi(\widetilde{A})^{\frac{p-q}{pq}}
\|f|{L}_{p}^{1}(\widetilde{\Omega})\|,\,\,\,1\leq q<p<\infty,
\end{equation}
holds for any function $f\in {L}_{p}^{1}(\widetilde{\Omega})\cap C_0(\widetilde{A})$.

Similarly to the proof of Theorem~\ref{thm:AnalyticPP}, we fix a cutoff function $\eta\in C_{0}^{\infty}(\mathbb R^n)$,
which is equal to one on the ball $B(0,1)$ and is equal to zero outside of the ball $B(0,2)$. Then, given $r>0$, define test functions
$f_j$ as in (\ref{test1}). By the computation performed in (\ref{test1}), we have 
\begin{equation}
|\nabla( f_j(y))|\leq C_{\eta}\,\,\text{for all}\,\,y\in\mathbb R^n,
\nonumber
\end{equation}
where the constant $C_{\eta}$ is as in (\ref{test3}). Notably, $C_{\eta}$ is independent of $r$.

By substituting the test functions $f_j$ into the inequality (\ref{eqphi}), we see that
\begin{equation}\label{EqFun}
\biggl(\int\limits_{\varphi^{-1}(B(y_0,r))}|D\varphi(x)|^q\,dx\biggr)^{1/q}\leq
C\Phi(B(y_0,2r))^{\frac{p-q}{pq}}(r^n)^{1/p}.
\end{equation}
Hence the homeomorphism $\varphi$ belongs to the Sobolev space $L_{q,{\loc}}^1(\Omega)$. 

We then prove that $\varphi:\Omega\to\widetilde{\Omega}$ is a homeomorphism of finite distortion. As in the proof of the previous theorem, we let $Z=\{x\in \Omega : J(x,\varphi)=0\}$, and let $S$ be the set from the change of variables formula (\ref{chvf}) where the homeomorphism $\varphi$ does not have the Luzin $N$-property. Reasoning as in the proof of Theorem~\ref{thm:AnalyticPP}, it
suffices to show that
$$
\int\limits_{Z\setminus S}|D\varphi(x)|^q\,dx=0.
$$
Next, we again let $\varepsilon>0$, and use the fact that $|\varphi(Z\setminus S)|=0$ by
the change of variables formula (\ref{chvf}) to select a cover $\{B(y_i,r_i)\}$
of $\varphi(Z\setminus S)$ such that the multiplicity of the covering $B(y_i,2r_i)$ is finite and $\sum\limits_i|B(y_i,r_i)|<\varepsilon$.

Then, by inequality (\ref{EqFun}), we obtain
\begin{multline}
\int\limits_{Z\setminus S}|D\varphi(x)|^q\,dx\leq
\sum\limits_{i=1}^{\infty}\int\limits_{\varphi^{-1}(B(y_i,r_i))}|D\varphi(x)|^q\,dx\\
\leq C\sum\limits_{i=1}^{\infty}\Phi(B(y_i,2r_i))^{\frac{p-q}{p}}(r_i^n)^{q/p}\\
\leq C\sum\limits_{i=1}^{\infty}\Phi(B(y_i,2r_i))^{\frac{p-q}{p}}
(\sum\limits_{i=1}^{\infty}r_i^n)^{q/p}.
\nonumber
\end{multline}
Since $\varepsilon$ is an arbitrary number, we have $\int\limits_{Z\setminus S}|D\varphi(x)|^q\,dx=0$. Hence we have that $D\varphi=0$ a.e. on $Z\setminus S$ and the homeomorphism $\varphi$ has finite distortion.

We then rewrite inequality~(\ref{EqFun}) to the form
$$
\biggl(
\int\limits_{\varphi^{-1}(B(y_0,r))}|D\varphi(x)|^q\,dx
\biggr)^{\frac{p}{p-q}}
\leq C\frac{\Phi(B(y_0,2r))}{|B(y_0,2r)|}(r^n)^{\frac{p}{p-q}}
$$
and apply to the left side of this inequality the change of variable formula, with the notation $B:=B(y_o,r)$:
\begin{multline}
\biggl(\int\limits_{\varphi^{-1}(B)}
|D\varphi(x)|^q\,dx\biggr)^{\frac{p}{p-q}}=
\biggl(\int\limits_{\varphi^{-1}(B)\setminus S}
|D\varphi(x)|^q\,dx\biggr)^{\frac{p}{p-q}}\\
=\biggl(\int\limits_{\varphi^{-1}(B)\setminus (S\cup Z)}
|D\varphi(x)|^q\,dx\biggr)^{\frac{p}{p-q}}=
\biggl(\int\limits_{\varphi^{-1}(B)\setminus (S\cup Z)}
\frac{|D\varphi(x)|^q}{|J(x,\varphi)|}|J(x,\varphi)|\,dx\biggr)^{\frac{p}{p-q}}\\
=\biggl(\int\limits_{B\setminus \varphi(S)}
\frac{|D\varphi(\varphi^{-1}(y))|^q}{|J(\varphi^{-1}(y),\varphi)|}\,dy\biggr)^{\frac{p}{p-q}}
\leq C\frac{\Phi(B(y_0,2r))}{|B(y_0,2r)|}(r^n)^{\frac{p}{p-q}}.
\nonumber
\end{multline}
Hence, we obtain the inequality
$$
\biggl(\frac{1}{r^n}\int\limits_{B(y_0,r)\setminus \varphi(S)}
\frac{|D\varphi(\varphi^{-1}(y))|^q}{|J(\varphi^{-1}(y),\varphi)|}\,dy\biggr)^{\frac{p}{p-q}}
\leq C\frac{\Phi(B(y_0,2r))}{|B(y_0,2r)|}.
$$

Using the Lebesgue theorem about differentiability of the integral and properties of the volume derivative of countably-additive set functions \cite{RR55,VU04} we obtain
$$
\biggl(
\frac{|D\varphi(\varphi^{-1}(y))|^q}{|J(\varphi^{-1}(y),\varphi)|}\biggr)^{\frac{p}{p-q}}
\leq C\Phi'(y)\quad\text{ for almost all}\,\,\, y\in\widetilde{\Omega}\setminus\varphi(S).
$$
By integrating the last inequality on an arbitrary open bounded subset $\widetilde{U}\subset\widetilde{\Omega}$, we obtain
\begin{multline*}
\int\limits_{\widetilde{U}\setminus\varphi(S)}\biggl(
\frac{|D\varphi(\varphi^{-1}(y))|^q}{|J(\varphi^{-1}(y),\varphi)|}\biggr)^{\frac{p}{p-q}}~dy
\leq C\int\limits_{\widetilde{U}\setminus\varphi(S)}\Phi'(y)~dy\\
\leq C\int\limits_{\widetilde{U}}\Phi'(y)~dy
\leq C \Phi(\widetilde{U})\leq C  Q_{p,q}^{\frac{p-q}{pq}}(\varphi;\Omega).
\end{multline*}
Since the choice of $\widetilde{U}\subset\widetilde{\Omega}$ is arbitrary, we have
$$
\int\limits_{\widetilde{\Omega}\setminus\varphi(S)}\biggl(
\frac{|D\varphi(\varphi^{-1}(y))|^q}{|J(\varphi^{-1}(y),\varphi)|}\biggr)^{\frac{p}{p-q}}~dy
\leq C Q_{p,q}^{\frac{p-q}{pq}}(\varphi;\Omega).
$$
Hence,
\begin{multline}
\int\limits_\Omega\left(\frac{|D\varphi(x)|^p}
{|J(x,\varphi)|}\right)^{\frac{q}{p-q}}\,dx=
\int\limits_\Omega\frac{|D\varphi(x)|^{\frac{pq}{p-q}}}
{|J(x,\varphi)|^{\frac{p}{p-q}}}|J(x,\varphi)|\,dx\\
=
\int\limits_{\widetilde{\Omega}\setminus\varphi(S)}\biggl(
\frac{|D\varphi(\varphi^{-1}(y))|^q}{|J(\varphi^{-1}(y),\varphi)|}\biggr)^{\frac{p}{p-q}}~dy
\leq C Q_{p,q}^{\frac{p-q}{pq}}(\varphi;\Omega)<\infty.
\nonumber
\end{multline}

{\it Sufficiency.}
Let $f\in L_p^1(\widetilde{\Omega})\cap C^\infty(\widetilde{\Omega})$,  then the composition $f\circ\varphi$
belongs to the class $W^1_{1,\loc}(\Omega)$ and the chain rule holds \cite{Zi}. So, we have
\begin{multline}
\|\varphi^* f\vert L_q^1(\Omega)\|=\biggl(\int\limits_\Omega|\nabla (f\circ\varphi)|^q\,dx\biggr)^{\frac{1}{q}}\\
\leq
\biggl(\int\limits_\Omega(|\nabla f||D\varphi(x)|)^q\,dx\biggr)^{\frac{1}{q}}
=\biggl(\int\limits_\Omega|\nabla f|^q|J(x,\varphi)|^{\frac{q}{p}}
\frac{|D\varphi(x)|)^q}{|J(x,\varphi)|^{\frac{q}{p}}}\,dx\biggr)^{\frac{1}{q}}.
\nonumber
\end{multline}
Using the H\"older inequality we obtain
\begin{multline}
\|\varphi^*f\vert L_q^1(\Omega)\|=\biggl(\int\limits_\Omega|\nabla (f\circ\varphi)|^q\,dx\biggr)^{\frac{1}{q}}\\
\le
\biggl(\int\limits_\Omega|\nabla f|^p(\varphi(x))\cdot|J(x,\varphi)|\,dx\biggr)^{\frac{1}{p}}
\cdot\biggl(\int\limits_\Omega\left(\frac{|D\varphi(x)|^p}
{|J(x,\varphi)|}\right)^{\frac{q}{p-q}}\,dx\biggr)^{\frac{p-q}{pq}}.
\nonumber
\end{multline}

Now the application of the change of variable formula gives the required inequality
$$
\|\varphi^{\ast}f \mid L^1_q(\Omega)\|\leq K_{p,q}(\varphi;\Omega) \|f \mid L^1_p(\widetilde{\Omega})\|
$$
for every smooth function $f\in L^1_p(\widetilde{\Omega})$.

The estimate is extended to all functions $f\in L^1_p(\widetilde{\Omega})$, $1<q< p<\infty$,, in a manner similar to how this was done in the proof of Theorem~\ref{thm:AnalyticPP}. Indeed, Corollary~\ref{cor:Capacity} again yields that if $f_k\in L^1_p(\widetilde{\Omega})\cap C^{\infty}(\widetilde{\Omega})$ tend to $f$ both in $L^1_p(\widetilde{\Omega})$ and $p$-quasi-everywhere
in $\widetilde{\Omega}$, then $\varphi^{\ast}(f_k)\to \varphi^{\ast}(f)$ $q$-quasi-everywhere in $\Omega$. Thus, the
composition operator $\varphi^{\ast}$ defined by $\varphi$ is given by its continuous extension from $L^1_p(\widetilde{\Omega})\cap C^{\infty}(\widetilde{\Omega})$ to $L^1_p(\widetilde{\Omega})$.

If $1 = q<p$, consider a sequence of smooth functions $f_k\in L^1_p(\widetilde{\Omega})$, $k=1,2,...$, such that $f_k\to f$ in $L^1_p(\widetilde{\Omega})$ and $f_k\to f$ $p$-quasi-everywhere in $\widetilde{\Omega}$ as $k\to\infty$. 
Then there exists a subsequence $\varphi^{\ast}(f_{n_k})$ which converges almost everywhere in $\Omega$. Hence the extension by continuity of the operator $\varphi^{\ast}$ $L^1_p(\widetilde{\Omega})\cap C^{\infty}(\widetilde{\Omega})$ to $L^1_p(\widetilde{\Omega})$ coincides with the operator of substitution $\varphi^{\ast}$, $\varphi^{\ast}(f) = f\circ\varphi$.
\end{proof}

\vskip 0.3cm
In the case  $1\leq q\leq p=\infty$ we need the following lemma \cite{GU10-2}:

\begin{lem}
\label{lem:WeakConv}
Let a function $f: \Omega\to \mathbb R^n$ belong to the Sobolev space $L^1_{\infty}(\Omega)$. Then there exists a sequence of smooth functions $\{f_k\}$ such that the sequence $\nabla f_k$ weakly converges to $\nabla f$ in $L_{\infty}(\Omega)$.
\end{lem}

By using this lemma we prove the following theorem \cite{GU10-2}.

\begin{thm} \cite{GU10-2}
\label{thm:AnalyticInfP}
Let $\varphi:\Omega\to \widetilde{\Omega}$ be a homeomorphism between two domains $\Omega,\widetilde{\Omega} \subset \mathbb R^n$. 
Then $\varphi$ belongs to the Sobolev space $L^1_q(\Omega)$, $1\leq q\leq \infty$ if and only if the composition operator
$$
\varphi^{\ast}: L^1_{\infty}(\widetilde{\Omega})\to L^1_{q}(\Omega)
$$
is bounded.
\end{thm}

\begin{proof}
{\it Necessity.} 
Since the composition operator $\varphi^{\ast}$ is bounded, substituting into the inequality
$$
\|\varphi^{\ast}f \mid L^1_q(\Omega)\|\leq \|\varphi^{\ast}\|\|f\mid L^1_{\infty}(\widetilde{\Omega})\|
$$
the coordinate functions $f_j=y_j$, $j=1,...,n$, we obtain that
$$
\|\varphi_j\mid L^1_q(\Omega)\|\leq \|\varphi^{\ast}\|, \,\,\,j=1,...,n,
$$
because $\|y_j\mid L^1_{\infty}(\widetilde{\Omega})\|=1$. Hence, the mapping $\varphi$ belongs to the Sobolev space $L^1_p(\Omega)$.

{\it Sufficiency.} Let $f$ be a smooth function of the class $L^1_{\infty}(\widetilde{\Omega})$. Then the composition $f\circ\varphi$ belongs to the Sobolev space $W^1_{1,\loc}(\Omega)$ and
\begin{multline}
\|\varphi^{\ast}f\mid L^1_q(\Omega)\|=\biggl(\int\limits_\Omega|\nabla(f\circ\varphi)|^q~dx\biggr)^{\frac{1}{q}}\leq
\biggl(\int\limits_\Omega|D\varphi(x)|^q|\nabla f|^q(\varphi(x))~dx\biggr)^{\frac{1}{q}}\\
\leq \biggl(\int\limits_\Omega|D\varphi(x)|^q~dx\biggr)^{\frac{1}{q}} \|f\mid L^1_{\infty}(\widetilde{\Omega})\|,
\nonumber
\end{multline}
where 
$$
K_{\infty,q}(\varphi;\Omega)=\biggl(\int\limits_\Omega|D\varphi(x)|^q~dx\biggr)^{\frac{1}{q}}<\infty.
$$

For an arbitrary function $f\in L^1_{\infty}(\widetilde{\Omega})$ we find by Lemma~\ref{lem:WeakConv} a sequence of smooth functions $f_k\in L^1_{\infty}(\widetilde{\Omega})$ such that $\nabla f_k\to \nabla f$ weakly in $L_{\infty}(\widetilde{\Omega})$ and 
$$
\|f_k\mid L^1_{\infty}(\widetilde{\Omega})\|\to\|f\mid L^1_{\infty}(\widetilde{\Omega})\|,
$$
see e.~g. \cite{Bu98}.
Note that 
$$
\|f\mid L^1_{\infty}(\widetilde{\Omega})\|\leq\liminf\limits_{k\to\infty} \|f_k\mid L^1_{\infty}(\widetilde{\Omega})\|,
$$
see e.~g. Lemma II. 3.27 in \cite{DSh58}.
Then $\nabla(\varphi^{\ast}f_k)\to \nabla(\varphi^{\ast}f)$ weakly in $L_p(\Omega)$ and
$$
\|\nabla(\varphi^{\ast}f) \mid L_p(\Omega)\|\leq \liminf\limits_{k\to \infty}\|\nabla(\varphi^{\ast}f_k)\mid L_{p}(\Omega)\|=\liminf\limits_{k\to \infty}\|\varphi^{\ast}f_k\mid L^1_p(\Omega)\|.
$$
Hence, tending to the limit in the inequality
$$
\|\varphi^{\ast}f_k\mid L^1_p(\Omega)\|\leq K_{\infty,q}(\varphi;\Omega) \|f_k\mid L^1_{\infty}(\widetilde{\Omega})\|
$$
we obtain that
$$
\|\varphi^{\ast}f\mid L^1_p(\Omega)\|\leq K_{\infty,q}(\varphi;\Omega) \|f\mid L^1_{\infty}(\widetilde{\Omega})\|,
$$
for any function $f\in L^1_{\infty}(\widetilde{\Omega})$.
\end{proof}

\subsection{Mappings of bounded $(p,q)$-distortion}

Let us define the "normalized" $p$-dilatation of a Sobolev homeomorphism $\varphi:\Omega\to\widetilde{\Omega}$ at a point $x\in\Omega$ as
$$
K_p(x)=\inf \{k(x): |D\varphi(x)|\leq k(x) |J(x,\varphi)|^{\frac{1}{p}},\,\,x\in\Omega \}.
$$
In the case $p=n$ this characteristic $K_p$ is the "normalized" conformal dilatation and in the case $p\ne n$ the  equivalent $p$-dilatation was introduced in \cite{Ge69}. 
Then $\varphi:\Omega\to\widetilde{\Omega}$ is called a mapping of bounded weak $(p,q)$-distortion \cite{UV10} or a weak $(p,q)$-quasiconformal mapping \cite{GGR95,VU98} if $\varphi\in W^1_{q,\loc}(\Omega)$, has finite distortion,
and 
\[
K_{p,q}(\varphi;\Omega)=\|K_p \mid L_{\kappa}(\Omega)\|<\infty,
\]
where $1/q-1/p=1/{\kappa}$ ($\kappa=\infty$, if $p=q$).

Let us consider as an example mappings of bounded weak $(p,q)$-distortion of Lipschitz domains onto anisotropic H\"older singular domains (introduced in \cite{GGu94}).
Let $H_g$ be a domain with anisotropic H\"older $\gamma$-singularities:
$$
H_g=\{ x\in\mathbb R^n : 0<x_n<1, 0<x_i<g_i(x_n),
\,i=1,2,\dots,n-1\}.
$$
Here $g_i(\tau)=\tau^{\gamma_i}$, $\gamma_i\geq 1$, $0\leq\tau\leq 1$ are H\"older functions. For the function $G=\prod_{i=1}^{n-1}g_i$ we denote 
$$
\gamma=\frac{\log G(\tau)}{\log \tau}+1.
$$
It is evident that $\gamma\geq n$. In the case $g_1=g_2=\dots=g_{n-1}$ we will say that the domain $H_g$ is a domain with $\sigma$-H\"older singularity, where $\sigma=(\gamma-1)/(n-1)$.
For $g_1(\tau)=g_2(\tau)=\dots=g_{n-1}(\tau)=\tau$ we use notation $H_1$ instead
of $H_g$.

We define the mapping $\varphi_a: H_1\to H_g$, $a>0$, by
$$
\varphi_a(x)=\left(\frac{x_1}{x_n}g^a_1(x_n),\dots,\frac{x_{n-1}}{x_n}g^a_{n-1}(x_n),x_n^a\right).
$$

\begin{thm}
\label{lemhol}
Let $(n-p)/(\gamma-p)<a<p(n-q)/q(\gamma-p)$. Then the mapping $\varphi_a: H_1\to H_g$, 
is a weak $(p,q)$-quasiconformal mapping, $1<q<p<\gamma$, $1<q<n$, from the Lipschitz convex domain $H_1$ onto the "cusp" domain $H_g$ with
$$
K_{p,q}(H_1)\leq a^{-\frac{1}{p}} c(p,q,\gamma)\,\sqrt{\sum_{i=1}^{n-1}(a\gamma_i-1)^2+n-1+a^2}\,
$$
where $c(p,q,\gamma)=\left(\frac{p-q}{np-q(a(\gamma-p)+p)}\right)^{(p-q)/pq}$.

\end{thm}

\begin{proof}
By simple calculations 
$$
\frac{\partial(\varphi_a)_i}{\partial
x_i}=\frac{g^a_i(x_n)}{x_n},\quad
\frac{\partial(\varphi_a)_i}{\partial
x_n}=\frac{-x_ig^a_i(x_n)}{x_n^{2}}+\frac{ax_ig^{a-1}_i(x_n)}{x_n}g'_i(x_n)
\quad\text{and}\quad\frac{\partial(\varphi_a)_n}{\partial
x_n}=ax_n^{a-1}
$$
for any $i=1,...,n-1$. Hence $J(x,\varphi_a)=ax_n^{a-n}G^a(x_n)=ax_n^{a\gamma-n}$, $J(x,\varphi_a)\leq a$ for $a>1$ and
\begin{multline}\label{maps}
D\varphi_a (x)=
\left(\begin{array}{cccc}
x_n^{a\gamma_1-1} & 0 & ... & (a\gamma_1-1)x_1x_n^{a\gamma_1-2}\\
0 & x_n^{a\gamma_2-1} & ... & (a\gamma_2-1)x_2x_n^{a\gamma_2-2}\\
... & ... & ... & ...\\
0 & 0 & ... & ax_n^{a-1}
\end{array} \right)\\
=
x_n^{a-1}\left(\begin{array}{cccc}
x_n^{a\gamma_1-a} & 0 & ... & (a\gamma_1-1)\frac{x_1}{x_n}x_n^{a(\gamma_1-1)}\\
0 & x_n^{a\gamma_2-a} & ... & (a\gamma_2-1)\frac{x_2}{x_n}x_n^{a(\gamma_2-1)}\\
... & ... & ... & ...\\
0 & 0 & ... & a
\end{array} \right).
\end{multline}

Because $0<x_n<1$ and $x_1/x_n<1$ we have that
\begin{equation}
|D\varphi_a(x)|\leq x_n^{a-1}\sqrt{\sum_{i=1}^{n-1}(a\gamma_i-1)^2+n-1+a^2}.
\nonumber
\end{equation}

Then 
\begin{multline*}
K_{p,q}(H_1)=\left(\int\limits_{H_1}\left(\frac{|D\varphi_a(x)|^p}{J(x,\varphi_a)}\right)^{\frac{q}{p-q}}~dx\right)^{\frac{p-q}{pq}}
\\
\leq  \frac{\sqrt{\sum_{i=1}^{n-1}(a\gamma_i-1)^2+n-1+a^2}}{\sqrt[p]{a}}\left(\int\limits_{H_1}x_n^{\frac{(p(a-1)-(a\gamma-n))q}{p-q}}~dx\right)^{\frac{p-q}{pq}}\\
=
\frac{\sqrt{\sum_{i=1}^{n-1}(a\gamma_i-1)^2+n-1+a^2}}{\sqrt[p]{a}} \left(\int\limits_0^1\int\limits_0^{x_n}...\int\limits_0^{x_n}x_n^{\frac{(p(a-1)-(a\gamma-n))q}{p-q}}~dx_1...dx_n\right)^{\frac{p-q}{pq}}\\=
\frac{\sqrt{\sum_{i=1}^{n-1}(a\gamma_i-1)^2+n-1+a^2}}{\sqrt[p]{a}}\left(\int\limits_0^1 x_n^{\frac{(p(a-1)-(a\gamma-n))q}{p-q}+n-1}~dx_n\right)^{\frac{p-q}{pq}}.
\end{multline*}
By the calculation
\begin{equation}
\label{int}
\left(\int\limits_0^1 x_n^{\frac{(p(a-1)-(a\gamma-n))q}{p-q}+n-1}~dx_n\right)^{\frac{p-q}{pq}}=\left(\frac{p-q}{np-q(a(\gamma-p)+p)}\right)^{\frac{p-q}{pq}},
\end{equation}
if $np-q(a(\gamma-p)+p>0$, or equivalently $a<p(n-q)/q(\gamma-p)$.

Hence 
$$
K_{p,q}(H_1) \leq c(p,q,\gamma) \frac{\sqrt{\sum_{i=1}^{n-1}(a\gamma_i-1)^2+n-1+a^2}}{\sqrt[p]{a}},
$$
if $a<p(n-q)/q(\gamma-p)$.

We then check that $1<q<np/(p+a\gamma-pa)<p$. The inequality $1<np/(p+a\gamma-pa)$ implies $a<(np-p)/(\gamma-p)$, and
$$
\frac{p(n-q)}{q(\gamma-p)}<\frac{np-p}{\gamma-p}, \,\,\text{if}\,\,q>1.
$$

Since the inequality $np/(p+a\gamma-pa)<p$ implies $a>(n-p)/(\gamma-p)$, we have that $a\in \left((n-p)/(\gamma-p),p(n-q)/q(\gamma-p)\right)$.
\end{proof}

\section{On the weak inverse mapping theorem for Sobolev homeomorphisms}

\subsection{The weak regularity of Sobolev homeomorphisms.} 

In this section  we consider weak regularity properties of mappings inverse to homeomorphisms that generate bounded composition operators on Sobolev spaces. On this base we prove that if $\varphi:\Omega\to\widetilde{\Omega}$ is a homeomorphism which generates a bounded composition operator
$$
\varphi^{\ast}: L^1_p(\widetilde{\Omega})\to L^1_q(\Omega),\,\,\,n-1\leq q\leq p\leq\infty,
$$
then the inverse mapping $\varphi^{-1}:\widetilde{\Omega}\to\Omega$ also generates a bounded composition operator on Sobolev spaces, see \cite{U93} (the limit case in \cite{GU10-2}).

Let us recall that a ring condenser $R$ (see, for example, \cite{Ge69,MRV69}) in a domain $\Omega\subset \mathbb R^n$ is a pair $R=(F;G)$ of sets $F\subset G\subset\Omega$, where $F$ is a connected closed set and $G$ is a connected open set.  
The $p$-capacity of a ring $R=(F;G)$ is defined by
$$
\cp_p(F;G) =\inf\|f\vert L_p^1(\Omega)\|^p,
$$
where the greatest lower bound is taken over all continuous functions $f\in L_p^1(\Omega)$ with a compact support in $G$ and such that $f\geq 1$ on $F$. Such functions are called admissible functions for the ring $R=(F;G)$. Note that our notation $R = (F; G)$ for ring condensers coexists with the notation $(F_1, F_2)$ for condensers defined in Section 3.1; in particular, we have $(F; G) = (F, \Omega \setminus G)$.

Recall the capacity estimate \cite{K86}:

\begin{lem} \label{lemcap} Let $E$ be a connected closed subset of
an open bounded set $G\subset\mathbb R^n, n\geq 2$, and $n-1<p<\infty$. Then
$$
\cp_p^{n-1}(E;G)\geq c\frac{(\diam E)^p}{|G|^{p-n+1}},
$$
where the constant $c$ depends on $n$ and $p$ only.
\end{lem}

By using this lemma we prove weak differentiability of mappings inverse to homeomorphisms generating bounded composition operators on Sobolev spaces \cite{U93}.

\begin{thm}\label{reginv}
Let a homeomorphism $\varphi :\Omega\to \widetilde{\Omega}$
generate a bounded composition operator
$$
\varphi^{\ast}: L^1_p(\widetilde{\Omega})\to L^1_q(\Omega),\,\,\,n-1<q\leq p<\infty.
$$
Then the inverse mapping $\varphi^{-1} :\widetilde{\Omega}\to \Omega$
is an $\ACL$-mapping, differentiable almost everywhere in
$\widetilde{\Omega}$.
\end{thm}

\begin{proof}
Consider an arbitrary $n$-dimensional cube $P,\,\overline{P}\subset \widetilde{\Omega}$, with edges parallel to coordinate axes. We show that $\varphi^{-1}$ is absolutely continuous on almost all lines in $P$ parallel to the $y_n$-axis.
Let $P_0$ be a projection of $P$ to the subspace $y_n=0$ and let $I$ be a projection of $P$ to the $y_n$-axis. Then $P=P_0\times I$. 

We consider the set function $\Phi$ defined by Theorem~\ref{thm:SobolevAddFun}. This monotone countably additive set function $\Phi$ generates by the rule $\Phi(A,P)=\Phi(A\times I)$  a monotone countably additive set function defined on open subsets $A\subset P_0$. It is known that for almost all points $z\in P_0$ the value
$$
\overline{\Phi}'(z,P)=\varlimsup\limits_{r\to 0} \frac{\Phi(B^{n-1}(z,r),P)}{|B^{n-1}(z,r)|}
$$
is finite, where $B^{n-1}$ is the $(n-1)$-dimensional ball with the center at the point $z$ and radius $r>0$ and
$|B^{n-1}(z,r)|$ its $(n-1)$-dimensional measure.

The $n$-dimensional Lebesgue measure defines a quasiadditive set function $\Psi(E) = \lvert \varphi^{-1}(E) \rvert$, where $E$ ranges over open sets in $\widetilde{\Omega}$. This set function generates the quasiadditive function $\Psi(A,P)=\Psi(A\times I)$ of open sets $A\subset P_0$,
and $\overline{\Psi}'(z,P)<\infty$ for almost all $z\in P_0$.

Fix an arbitrary point $z\in P_0$, at which the derivatives $\overline{\Phi}{}'(z,P)$ and $\overline{\Psi}'(z,P)$ are finite.
On the section $I_z=\{z\}\times I$ of the cube $P$, we take $\gamma_1,\dots ,\gamma_k$ mutually disjoint intervals
with lengths $l(\gamma_1),\dots ,l(\gamma_k)$ correspondingly. 

Let us denote by $R_i$ the set of points with distance less than $r$ from $\gamma_i$. We consider ring condensers $(\gamma_i,F_i),\,F_i=
\widetilde{\Omega}\backslash(\gamma _i\cup R_i)$. We choose a number $r > 0$ such that the sets $R_1,\dots ,R_k$ are mutually disjoint, $R_i\subset P$ and $\omega_nr<\omega_{n-1}l(\gamma_i),\,i=1,\dots ,k$, where $\omega_n$ and
$\omega_{n-1}$ are the volume of the unit ball and the area of its surface in $\mathbb R^n$ correspondingly. 

Now, applying the inequality
$$
\cp_{q}^{1/q}(\varphi^{-1}(F_0),\varphi^{-1}(F_1);\Omega)
\leq\Phi(\widetilde{\Omega}\setminus(F_0\cup F_1))^{\frac{p-q}{pq}}
\cp_{p}^{1/p}(F_0,F_1;\widetilde{\Omega}),
$$
where we set $\Phi(\widetilde{\Omega}\setminus(F_0\cup F_1))^{\frac{p-q}{pq}}:=K_{p,p}(\varphi;\Omega)$ in the case $p=q$,
the upper capacity estimate \cite{GResh}
$$
{{\cp}}_p(\gamma_i,F_i;\widetilde{\Omega})\le\frac{|R_i|}{r^p}
\le 2\omega_{n-1}b_ir^{n-1-p},
$$
and the lower capacity estimate
$$
{{\cp}}_q(\varphi^{-1}(\gamma_i),\varphi^{-1}(F_i);\Omega)\ge
c\frac{(\diam\varphi^{-1}(\gamma_i))^{q/(n-1)}}
{|\varphi^{-1}(R_i)|^{(1-n+q)/(n-1)}},
$$
given by Lemma~\ref{lemcap}, we obtain
$$
\diam\varphi^{-1}(\Delta_i)
\le c_1\cdot r^{(n-1-p)(n-1)/p}
\Phi(R_i)^{(p-q)(n-1)/pq}\Psi(R_i)^{(1-n+q)/q}b_i^{(n-1)/p}.
$$
Summing overy $i=1,\dots ,k$, applying the H\"older inequality and using the definition of the countably additive set function, we have
\begin{multline}
\sum\limits_{i=1}^{k}\diam\varphi^{-1}(\Delta_i)\le c_1
\left(\frac{\Phi(B^{n-1}(z,r),P)}{|B^{n-1}(z,r)|}\right)^
{((p-q)/(n-1))/pq}\\
\times\left(\frac{\Psi(B^{n-1}(z,r),P)}{|B^{n-1}(z,r)|}\right)^
{(1-n+q)/q}\left(\sum\limits_{i=1}^{k}l(\gamma_i)\right)^{(n-1)/p}.
\nonumber
\end{multline}
Passing to the limit $r\to 0$, we obtain
$$
\sum\limits_{i=1}^k\diam\varphi^{-1}(\gamma_i)\le c_2
\left(\sum\limits_{i=1}^k l(\gamma_i)\right)^{(n-1)/p},
$$
which implies that $\varphi^{-1}\in {\ACL}(\widetilde{\Omega})$.

Now, we show that $\varphi^{-1}$ is differentiable almost everywhere in $\widetilde{\Omega}$.
For every point $y\in \widetilde{\Omega}$ we consider the ring condenser
$(F_0,F_1)$, where
$F_0=\{z\vert\,|y-z|\le r\},\ F_1=\{z\vert\,|y-z|\ge 2r\}$.
Then ${{\cp}}_p(F_0,F_1;\widetilde{\Omega})=c\cdot r^{n-p}$, where $c$ is some constant.
Using the lower capacity estimates we obtain
$$
\left(\frac{\diam\varphi^{-1}(F_0)}r\right)^{pq}
\le c\left(\frac{\Phi(B)}{|B|}\right)^{(n-1)(p-q)}
\left(\frac{\Psi(B)}{|B|}\right)^{pq+p-np}.
$$
Here, as before, $\Psi(B)=|\varphi^{-1}(B)|$,
$B=\widetilde{\Omega}\backslash  F_1$. Passing to the limit while $r\to 0$, for almost all points
$y\in \widetilde{\Omega}$, we obtain
$$
L(y,\varphi^{-1})\le c(\overline{\Phi}'(y))^{(n-1)(p-q)/pq}
(\overline{\Psi}'(y))^{(1-n+q)/q},
$$
where
$$
L(y,\varphi^{-1})=\varlimsup_{r\to 0}(\max_{|y-z|=r}|\varphi^{-1}(y)-
\varphi^{-1}(z)|/r),
$$
and the values $\overline{\Phi}'=\varlimsup\limits_{r\to 0}(\Phi(B)/|B|)$
and $\overline{\Psi}'(y)=\varlimsup\limits_{r\to 0}(\Psi(B)/|B|)$
are finite almost everywhere in $\widetilde{\Omega}$. By using Stepanov's theorem, \cite{Fe69} we obtain that $\varphi^{-1}$
is differentiable almost everywhere in $\widetilde{\Omega}$.
\end{proof}

By using this theorem, we obtain the following duality property of composition operators on Sobolev spaces \cite{U93}. 

\begin{thm}
\label{CompDual} Let a homeomorphism $\varphi:\Omega\to\widetilde{\Omega}$,  $\Omega,\widetilde{\Omega}\subset\mathbb R^n$, generate a bounded composition operator 
\[
\varphi^{\ast}:L^1_p(\widetilde{\Omega})\to L^1_q(\Omega),\,\,\,n-1<q\leq p<\infty.
\]
Then the inverse mapping $\varphi^{-1}:\widetilde{\Omega}\to\Omega$ induces a bounded composition operator 
\[
\left(\varphi^{-1}\right)^{\ast}:L^1_{q'}(\Omega)\to L^1_{p'}(\widetilde{\Omega}),\,\,\,n-1<p'\leq q'<\infty,
\]
where $p'=p/(p-n+1)$ and $q'=q/(q-n+1)$.
\end{thm}

\begin{proof}
By Theorem~\ref{reginv} we have that $\varphi^{-1}\in W^1_{1,\loc}(\widetilde{\Omega})$. Then by \cite{P93} (see also \cite{CHM,GU10-1,HKM06}) we have
$$
|D\varphi^{-1}(y)|=
\begin{cases}
\biggl(\frac{|\adj D\varphi|(x)}{|J(x,\varphi)|}\biggr)_{x=\varphi^{-1}(y)}
& \text{if}\quad x\in \Omega\setminus \left(S\cup Z\right),
\\
\,\,0 & \text{otherwise}.
\end{cases}
$$
Hence,
\[
|D\varphi^{-1}(y)|\leq\frac{|D\varphi(x)|^{n-1}}{|J(x,\varphi)|},
\]
for almost all $x\in \Omega\setminus \left(S\cup Z\right)$, $y=\varphi(x)\in \widetilde{\Omega}\setminus \varphi\left(S\cup Z\right)$,
and  
$$
|D\varphi^{-1}(y)|=0\,\,\text{for almost all}\,\,y\in \varphi(S).
$$

Now, taking into account that
\[
\frac{q'p'}{q'-p'}=\frac{pq}{(p-q)(n-1)}
\]
we obtain 
\[
\int\limits _{\widetilde{\Omega}}\left(\frac{|D\varphi^{-1}(y)|^{q'}}{|J(y,\varphi^{-1})|}\right)^{p'/(q'-p')}\, dy\le\int\limits _{\Omega}\left(\frac{|D\varphi(x)|^{p}}{|J(x,\varphi)|}\right)^{q/(p-q)}\, dx,
\]
where in the case $p = q$ we have $p' = q'$ and $L_\infty$-norms instead of integrals. Thus, by Theorems~\ref{thm:AnalyticPP} and \ref{thm:AnalyticPQ}, we have a bounded composition operator
\[
\left(\varphi^{-1}\right)^{\ast}:L^1_{q'}(\Omega)\to L^1_{p'}(\widetilde{\Omega}),\,\,\,n-1<p'\leq q'<\infty.
\]

\end{proof}

\begin{rem}
In the case $n=2$ we have $p'=p/(p-1)$, $q'=q/(q-1)$ and $p''=p$, $q''=q$. Hence the homeomorphism $\varphi:\Omega\to\widetilde{\Omega}$,  $\Omega,\widetilde{\Omega}\subset\mathbb R^n$, induces a bounded composition
operator 
\[
\varphi^{\ast}:L^1_p(\widetilde{\Omega})\to L^1_q(\Omega),\,\,\,1<q\leq p<\infty,
\]
if and only if the inverse mapping $\varphi^{-1}:\widetilde{\Omega}\to\Omega$ induces a bounded composition operator 
\[
\left(\varphi^{-1}\right)^{\ast}:L^1_{q'}(\Omega)\to L^1_{p'}(\widetilde{\Omega}),\,\,\,1<p'\leq q'<\infty.
\]

In the case $n\ne 2$ we have 
$$
p'' =\left(p'\right)'=\frac{p}{(n-1)^2-p(n-2)}\ne p, \,\,\text{if}\,\,p'>n-1,
$$
$$
q'' =\left(q'\right)'=\frac{q}{(n-1)^2-q(n-2)}\ne q, \,\,\text{if}\,\,q'>n-1,
$$
and this case is more complicated.
\end{rem}

The weak regularity of the inverse mapping in the case $q=n-1$ is more complicated and requires additional assumptions.

\begin{thm}
\label{inv_op}
Let a homeomorphism $\varphi : \Omega\to \widetilde{\Omega}$
between two domains $\Omega$ and $\widetilde{\Omega} \subset \mathbb R^n $, $n\geq 2$, generate by the composition rule $\varphi^{\ast}f=f\circ\varphi$ a bounded operator 
$$
\varphi^{\ast}: L^1_{\infty}(\widetilde{\Omega})\to L^1_{n-1}(\Omega).
$$
In addition, suppose that $\varphi$ possesses the Luzin $N$-property and has finite distortion. Then the inverse mapping $\varphi^{-1}: \widetilde{\Omega}\to \Omega$ generates the composition operator
$$
(\varphi^{-1})^{\ast}: L^1_{\infty}(\Omega)\to L^1_{1}(\widetilde{\Omega})
$$
and belongs to the Sobolev space $L^1_1(\widetilde{\Omega})$.
\end{thm}

\begin{proof}
Let $f\in L^1_{\infty}(\widetilde{\Omega})$. Then the composition $f\circ\varphi$ is weakly differentiable in $\Omega$.
Since $\varphi$ is a mapping of finite distortion, $D\varphi(x)=0$ at a.e. $x\in \Omega$ where $J(x,\varphi)=0$. Hence, we can define $\adj D\varphi(x)=0$ at such points.
Then  
$$
|J(x,\varphi)||\nabla f|(\varphi(x))\leq |\nabla(f\circ\varphi)|(x)\adj D\varphi(x) \,\,\,\text{for almost all}\,\,\,x\in \Omega,
$$
because 
$$
\min\limits_{|h|=1}|D\varphi(x)\cdot h|=\frac{1}{\max\limits_{|h|=1}|(D\varphi(x))^{-1}\cdot h|}
$$
and
$$
(D\varphi(x))^{-1}=J^{-1}(x,\varphi)\adj D\varphi(x)
$$
if $J(x,\varphi)\ne 0$.
Hence,
\begin{multline}
\|f\mid L^1_1(\widetilde{\Omega})\|=\int\limits_{\widetilde{\Omega}}|\nabla f|(y)~dy=\int\limits_{\Omega}|\nabla f|(\varphi(x))|J(x,\varphi)|~dx\\
\leq
\int\limits_{\Omega}|\nabla (f\circ\varphi)|(\varphi(x))|\adj D\varphi(x)|~dx
\leq 
\int\limits_{\Omega}|\nabla (f\circ\varphi)|(\varphi(x))|D\varphi|^{n-1}(x)~dx\\
\leq \|\varphi\mid L^1_{n-1}(\Omega)\|^{n-1}\cdot \|\varphi^{\ast}f\mid L^1_{\infty}(\Omega)\|
\nonumber
\end{multline}
because by Theorem~\ref{thm:AnalyticInfP} the mapping $\varphi$ belongs to $L^1_{n-1}(\Omega)$.
Thus, we have the lower estimate for the composition operator
$$
\|f\mid L^1_1(\widetilde{\Omega})\|\leq \|\varphi\mid L^1_{n-1}(\Omega)\|^{n-1}\cdot \|\varphi^{\ast}f\mid L^1_{\infty}(\Omega)\|
$$
whenever $\varphi^{\ast}f$ belongs to $L^1_{n-1}(\Omega)$.
Therefore, the inverse operator $(\varphi^{\ast})^{-1}=(\varphi^{-1})^{\ast}$ is a bounded operator 
$$
(\varphi^{-1})^{\ast} : L^1_{\infty}(\Omega)\cap L^1_{n-1}(\Omega)\to L^1_1(\widetilde{\Omega}).
$$
Fix a cutoff function $\eta\in C_0^{\infty}(B(0,2))$ such that $\eta=1$ on $B(0,1)$. By substituting into the inequality
$$
\|(\varphi^{-1})^{\ast}g\mid L^1_1(\widetilde{\Omega})\|\leq K_{\infty,1}(\varphi^{-1};\widetilde{\Omega})) \|g\mid L^1_{\infty}(\Omega)\|
$$
the test functions $g_j=(x_j-x_{0j})\eta\bigl(\frac{x-x_0}{r}\bigr)$, $x_0\in \Omega$, $r<\dist(x_0,\partial \Omega)$, we obtain that
$\varphi^{-1}$ belongs to the Sobolev space $L^1_{1,\loc}(\widetilde{\Omega})$.
Now, using the estimate
$$
\int\limits_{\widetilde{\Omega}}|D\varphi^{-1}(y)|~dy\leq\int\limits_{\widetilde{\Omega}}\frac{|D\varphi(\varphi^{-1}(y))|^{n-1}}{|J(\varphi^{-1}(y),\varphi)|}~dy=
\int\limits_{\Omega}|D\varphi(x)|^{n-1}~dx<+\infty
$$
we conclude that $\varphi^{-1}$ belongs to $L^1_1(\widetilde{\Omega})$ and by Theorem~\ref{thm:AnalyticInfP} generates a bounded composition operator from $L^1_{\infty}(\Omega)$ to $L^1_1(\widetilde{\Omega})$.
\end{proof}

By using Theorem~\ref{thm:AnalyticInfP} and Theorem~\ref{inv_op} we obtain the weak regularity of mappings which are inverse to Sobolev homeomorphisms.

\begin{thm}
\label{inv}
Let a homeomorphism $\varphi : \Omega\to \widetilde{\Omega}$ between two domains $\Omega$ and $\widetilde{\Omega} \subset \mathbb R^n $, $n\geq 2$, belong to the Sobolev space $L^1_{n-1}(\Omega)$, possess the Luzin $N$-property and have finite distortion. Then the inverse mapping $\varphi^{-1}: \widetilde{\Omega}\to \Omega$ belongs to the Sobolev space $L^1_1(\widetilde{\Omega})$.
\end{thm}

\subsection{Remarks on the weak regularity and the Luzin $N$-property.} 

Let us recall the following area formula \cite{Fe69, KMZ12}.
Let $\Omega,\widetilde{\Omega}$ be domains in $\mathbb R^n$, and let $\varphi : \Omega\to \widetilde{\Omega}$ be a homeomorphism of $W^1_{1,\loc}(\Omega)$. Then there exists a Borel set $S\subset \Omega$, $|S|=0$, such that the mapping $\varphi:\Omega\setminus S \to \widetilde{\Omega}\setminus\widetilde{S}$, $\widetilde{S}=\varphi(S)$, has the Luzin $N$-property and the area formula
\begin{equation}
\label{af}
\int\limits_{E\setminus S} f(x) |J(x,\varphi)|~dx=\int\limits_{\varphi(E)\setminus\widetilde{S}} \left(f(x)\vert_{x=\varphi^{-1}(y)}\right)~dy
\end{equation}
holds for every measurable set $E\subset \Omega$ and for every nonnegative measurable function $f: \Omega\to\mathbb R$. If a mapping $\varphi$ possesses the Luzin $N$-property (the image of a set of measure zero has measure zero), then $|\varphi (S)|=0$ and the second integral can be rewritten as the integral on $\varphi(E)$:
\begin{equation}
\label{afN}
\int\limits_{E} f(x) |J(x,\varphi)|~dx=\int\limits_{\varphi(E)} \left(f(x)\vert_{x=\varphi^{-1}(y)}\right)~dy.
\end{equation}
In \cite{KMZ12} it was proved that for Sobolev homeomorphisms of $W^1_{1,\loc}(\Omega)$, the area formula (\ref{afN}) is equivalent to the Luzin $N$-condition (the image of a set of measure zero has measure zero).

\begin{rem}
In the change of variables formula (\ref{af}) a measurable function $f: \Omega \to\mathbb R$ is an element of the space of measurable functions $\mathbb X( \Omega)$, i.e. an equivalence class up to a set of measure zero. The representation of measurable functions as equivalence classes is required, in particular, in Egorov's Theorem and Luzin's Theorem \cite{Fe69}.
\end{rem}

The consideration of special representations of measurable functions can lead to confusion. If, for example, we take in the area formula (\ref{af}) the special representation of a measurable function $f: \Omega\to\mathbb R$ such that $f=0$ on $S$, then $f\circ\varphi^{-1}=0$ on $\varphi(S)$, and
\begin{multline}
\label{afex}
\int\limits_{E} f(x) |J(x,\varphi)|~dx=\int\limits_{E\setminus S} f(x) |J(x,\varphi)|~dx=\int\limits_{\varphi(E)\setminus \varphi(S)} \left(f(x)\vert_{x=\varphi^{-1}(y)}\right)~dy
\\
=\int\limits_{\varphi(E)\setminus \varphi(S)} \left(f(x)\vert_{x=\varphi^{-1}(y)}\right)~dy+\int\limits_{\varphi(S)} \left(f(x)\vert_{x=\varphi^{-1}(y)}\right)~dy
=\int\limits_{\varphi(E)} \left(f(x)\vert_{x=\varphi^{-1}(y)}\right)~dy,
\end{multline}
that contradicts the conclusion of \cite{KMZ12}.

This area formula was used in \cite{CHM} for the proof of the following weak regularity theorem. 

\begin{thm}
\label{thm:chm}
Let $\Omega\subset\mathbb R^n$ be a domain and let $\varphi : \Omega\to \widetilde{\Omega}$ be a homeomorphism of $W^1_{n-1,\loc}(\Omega)$ of finite distortion. Then $\varphi^{-1}\in W^1_{1,\loc}(\widetilde{\Omega})$ and $\varphi^{-1}$ is a mapping of finite distortion.
\end{thm}

\vskip 0.2cm

However, the proof of Theorem~\ref{thm:chm} has the following problematic points, which are related to the measure theory.
\vskip 0.2cm
\noindent
{\bf 1.} The weak gradient $g: \widetilde{\Omega} \to \Omega$ of the inverse mapping $\varphi^{-1}: \widetilde{\Omega} \to \Omega$ is defined as the composition of the measurable co-distortion function 
$$
H_{\varphi}(x)=
\begin{cases}
\frac{|\rm{adj} D\varphi(x)|}{J(x,\varphi)}\,\,&\text{if}\,\,x\in \Omega\setminus S\,\,\text{and}\,\,J(x,\varphi)>0,\\
0  &\text{otherwise},
\end{cases}
$$ 
where $S$ is a set of measure zero from (\ref{chvf}), with the inverse mapping $\varphi^{-1}$. Unfortunately, because $\varphi$ does not have the Luzin $N$-property, the composition $g=H_{\varphi}\circ\varphi^{-1}$ can be non-measurable. However, it was not proven in \cite{CHM} that $g=H_{\varphi}\circ\varphi^{-1}$ is a measurable function.

\vskip 0.2cm
\noindent
{\bf 2.} Let us suppose, for simplicity's sake, that $J(x,\varphi) > 0$ for almost all $x \in \Omega$. In \cite{CHM}, by using the special choice of the measurable function $H_{\varphi}$: $H_{\varphi} = 0$ on the set of measure zero $S$, it was calculated that 
\begin{multline}
\label{calc_1}
\int\limits_{\varphi^{-1}(A)} |{\rm{adj}} D\varphi(x)|~dx=\int\limits_{\varphi^{-1}(A)} H_{\varphi}(x) |J(x,\varphi)|~dx=\int\limits_{\varphi^{-1}(A)\setminus S} H_{\varphi}(x) |J(x,\varphi)|~dx\\
=\int\limits_{A\setminus \varphi(S)} \left(H_{\varphi}(x)\vert_{x=\varphi^{-1}(y)}\right)~dy
=\int\limits_{A\setminus \varphi(S)} \left(H_{\varphi}(x)\vert_{x=\varphi^{-1}(y)}\right)~dy
\\
+\int\limits_{\varphi(S)} \left(H_{\varphi}(x)\vert_{x=\varphi^{-1}(y)}\right)~dy
=\int\limits_{A} g(y)~dy,
\end{multline}
for any compact set $A\subset\widetilde{\Omega}$, and it was stated that $g\in L^1_{1,\loc}(\widetilde{\Omega})$ because $\varphi\in W^1_{n-1,\loc}(\Omega)$.

\vskip 0.2cm
\noindent
However, this calculation depends on the choice of the representation of the measurable function $H_{\varphi}$.
Let us consider another representation of this measurable function:
$$
\widetilde{H}_{\varphi}(x) =
\begin{cases}
\frac{|\rm{adj} D\varphi(x)|}{J(x,\varphi)} & \text{if } x \in \Omega \setminus S \text{ and } J(x,\varphi) > 0, \\
1 & \text{otherwise}.
\end{cases}
$$ 
The functions $H_{\varphi}$ and $\widetilde{H}_{\varphi}$ are representations of the same measurable function. However, if we repeat the calculations from \cite{CHM} for $\widetilde{H}_{\varphi}$, we obtain (assuming that $\widetilde{g}$ is measurable) for the function $\widetilde{g} = \widetilde{H}_{\varphi} \circ \varphi^{-1}$:
\begin{multline}
\label{calc_2}
\int\limits_{\varphi^{-1}(A)} |{\rm{adj}} D\varphi(x)|~dx=\int\limits_{\varphi^{-1}(A)} \widetilde{H}_{\varphi}(x) |J(x,\varphi)|~dx=\int\limits_{\varphi^{-1}(A)\setminus S} \widetilde{H}_{\varphi}(x) |J(x,\varphi)|~dx\\
=\int\limits_{A\setminus \varphi(S)} \left(\widetilde{H}_{\varphi}(x)\vert_{x=\varphi^{-1}(y)}\right)~dy
<\int\limits_{A\setminus \varphi(S)} \left(\widetilde{H}_{\varphi}(x)\vert_{x=\varphi^{-1}(y)}\right)~dy
\\
+\int\limits_{\varphi(S)} \left(\widetilde{H}_{\varphi}(x)\vert_{x=\varphi^{-1}(y)}\right)~dy
=\int\limits_{A} \widetilde{g}(y)~dy,
\end{multline}
and so, the integrability of $|{\rm adj} D\varphi(x)|$ does not imply the integrability of $\widetilde{g}$. Thus, we have that the integrability of the suggested "gradient" of $\varphi^{-1}$ depends on the choice of the representation of the measurable co-distortion function $H_{\varphi}$, that does not correspond to the Lebesgue integral theory.

\vskip 0.2cm
\noindent
In addition, this method, which can be referred to as the "method of special representatives", leads to contradictions in geometric measure theory. Indeed, let $f: E \to \mathbb{R}$ be a measurable function. Let us consider a special representative of $f$ such that $f = 0$ on a set $S$ with $|S| = 0$. Then the area formula

\begin{equation*}
\int\limits_{E \setminus S} f(x) |J(x,\varphi)|~dx = \int\limits_{\varphi(E) \setminus \varphi(S)} \left(f(x)\vert_{x=\varphi^{-1}(y)}\right)~dy = \int\limits_{\varphi(E)} \left(f(x)\vert_{x=\varphi^{-1}(y)}\right)~dy,
\end{equation*}
holds without the Luzin $N$-property of the homeomorphism $\varphi$ for all measurable functions $f: E \to \mathbb{R}$. This seems unreasonable.

\vskip 0.2cm

Therefore, the weak regularity of Sobolev homeomorphisms without the Luzin $N$-property, does not have a complete rigorous proof in the article \cite{CHM}. However, it is possible that a rigorous proof could be obtained by modifying the proof given in the article \cite{CHM}.

\vskip1cm

\noindent
{\bf Conclusion.}
In the present review we have tried to give an introduction to the geometric theory of composition operators on Sobolev spaces. Unfortunately, numerous properties of weak $(p,q)$-quasiconformal mappings, like measure and metric distortions, sequences, boundary behavior and so on, have been omitted due to volume constraints. Thank you for reading of the article.

\vskip 1cm

\noindent
Vladimir Gol'dshtein  \,  \hskip 3.2cm Alexander Ukhlov

\noindent
Department of Mathematics   \hskip 2.25cm Department of Mathematics

\noindent
Ben-Gurion University of the Negev  \hskip 1.05cm Ben-Gurion University of the Negev

\noindent
P.O.Box 653, Beer Sheva, 84105, Israel  \hskip 0.7cm P.O.Box 653, Beer Sheva, 84105, Israel

\noindent
E-mail: vladimir@bgu.ac.il  \hskip 2.5cm E-mail: ukhlov@math.bgu.ac.il
\end{document}